%% file: main.tex
\documentclass[11pt,dvipsnames]{article}
\usepackage[letterpaper,margin=1in]{geometry}
\usepackage{enumitem}
\usepackage[noblocks]{authblk}
\setitemize{itemsep=0pt}

\input{includes.tex}

\input{macros.tex}

\title{A Universal Construction for Unique Sink Orientations}

\let\anonymous\undefined 

\ifdefined\anonymous
    \author{Redacted}
    \affil{Affiliation\\ \texttt{email@adre.ss}}
    \usepackage{lineno}
    \linenumbers
\else
    \author{Michaela~Borzechowski}
    \affil{Institut f\"ur Informatik, Freie Universit\"at Berlin\\ \texttt{michaela.borzechowski@fu-berlin.de}}
    \author{Joseph~Doolittle}
    \affil{Institut f\"ur Geometrie, Technische Universit\"at Graz\\
    \texttt{jdoolittle@tugraz.at}}
    \author{Simon~Weber}
    \affil{Department of Computer Science\\ ETH Zürich\\ \texttt{simon.weber@inf.ethz.ch}}
\fi

\date{}

\begin{document}

\maketitle
\thispagestyle{empty}
\setcounter{page}{0}
\begin{abstract}
Unique Sink Orientations (USOs) of cubes can be used to capture the combinatorial structure of many essential algebraic and geometric problems. 
For various structural and algorithmic questions, including enumeration of USOs and algorithm analysis, it is crucial to have systematic constructions of USOs.
While some construction methods for USOs already exist, each one of them has some significant downside.
Most of the construction methods have limited expressivity --- USOs with some desired properties cannot be constructed. In contrast, the phase flips of Schurr can construct all USOs, but the operation is not well understood.
We were inspired by techniques from cube tilings of space; we expand upon existing techniques in the area to develop \emph{generalized rewriting rules} for USOs.
These rewriting rules are a new construction framework which can be applied to all USOs.
The rewriting rules can generate every USO using only USOs of lower dimension.
The effect of any specific rewriting rule on an USO is simple to understand.
A special case of our construction produces a new elementary transformation of USOs, which we call a \emph{partial swap}.
We further investigate the relationship between partial swaps and phase flips and generalize partial swaps to \emph{phase swaps}.

\end{abstract}

\ifdefined\anonymous
\else
{\footnotesize
\vfill
\textbf{Acknowledgments.}
Michaela Borzechowski is supported by the German Research Foundation DFG within the Research Training Group GRK~2434 \emph{Facets of Complexity}.
Joseph Doolittle is supported by the Austrian Science Fund~FWF, grant P~33278.
Simon Weber is supported by the Swiss National Science Foundation~SNSF, within project~no.~204320.
}
\fi

\newpage
\section{Introduction}

A Unique Sink Orientation (USO) is an orientation of the hypercube graph, such that the induced subgraph of every non-empty subcube contains exactly one sink. \Cref{fig:spinner} shows an USO of the $3$-dimensional cube graph.

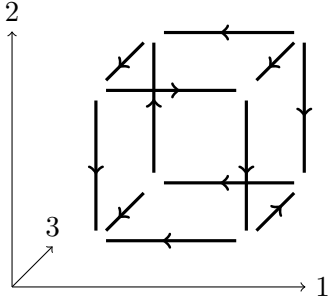
\begin{figure}[h!]
\centering
\begin{tikzpicture}[scale=2]
\newcommand{\dist}{0.5}
\draw[->] (-1.5*\dist,-1*\dist,1*\dist) -- (1.2, -1*\dist,1*\dist) node[right] {$1$};
\draw[->] (-1.5*\dist,-1*\dist,1*\dist) -- (-1.5*\dist, 1.2,1*\dist) node[above] {$2$};
\draw[->] (-1.5*\dist,-1*\dist,1*\dist) -- (-1.5*\dist, -1*\dist, -0.2) node[above] {$3$};
\node (filler) at (1+1.5*\dist,0,1) {};
\foreach \x in {0,1}{
\foreach \y in {0,1}{
\foreach \z in {0,1}{
\node (\x\y\z) at (\x, \y, {1-\z}) {};
}}}
\draw[very thick,postaction=-<-] (000) -- (100);
\draw[very thick,postaction=-<-] (000) -- (010);
\draw[very thick,postaction=-<-] (000) -- (001);
\draw[very thick,postaction=-<-] (011) -- (111);
\draw[very thick,postaction=-<-] (101) -- (111);
\draw[very thick,postaction=-<-] (110) -- (111);
\draw[very thick,postaction=-<-] (100) -- (110);
\draw[very thick,postaction=-<-] (101) -- (100);
\draw[very thick,postaction=-<-] (001) -- (101);
\draw[very thick,postaction=-<-] (011) -- (001);
\draw[very thick,postaction=-<-] (010) -- (011);
\draw[very thick,postaction=-<-] (110) -- (010);
\end{tikzpicture}
\caption{A Unique Sink Orientation of the cube. Note that USOs can also be cyclic.}
\label{fig:spinner}
\end{figure}

USOs were first defined by \szabo{} and Welzl in 2001~\cite{szabo2001usos}. 
They have been widely studied by researchers interested in various optimization problems, as they encode the combinatorial structure of these problems. Examples include the P-matrix linear complementarity problem, linear programming, and many more~\cite{gaertner2006lpuso,gaertner2001enforcing,klaus2012phd,schurr2004phd,stickney1978digraph}.
USOs have also attracted attention as purely combinatorial objects, with interest in structural and algorithmic directions ~\cite{bosshard2017pseudo,fearnley2020ueopl,gaertner2002simplex,gaertner2008grids,gaertner2015recognizing,gaertner2016niceusos,matousek2006numberusos,schurr2004quadraticbound}. 

On the structural side, enumerating or sampling USOs is a challenge. 
First off, recognizing USOs cannot be done in polynomial time in the dimension of the cube, even for orientations that can be represented by a polynomially-sized circuit~\cite{gaertner2015recognizing}, assuming \textsf{coNP} is not contained in \textsf{P}.
Furthermore, out of the many possible orientations for a given cube only few of them are USOs. 
But still, for a cube of a fixed dimension, there are doubly exponentially many USOs in that dimension~\cite{matousek2006numberusos}.
Construction techniques are used to generate bounds on the total number of USOs.
They are also useful to find counterexamples to suspected properties of USOs. 
On the algorithmic side, families of ``bad'' USOs provide examples to show lower bounds for the worst-case runtime of algorithms.
Experimental analysis of algorithms can benefit from efficient generation of random USOs.

Since there are $2^{\Theta(2^k\log k)}$ USOs of the $k$-dimensional hypercube~\cite{matousek2006numberusos} and even more generic orientations, brute-force enumeration quickly becomes an impractical strategy to generate USOs. 
Rather than brute-force search, more systematic approaches are needed to produce and modify~USOs.

\paragraph{Known Construction Methods.}
Already \szabo{} and Welzl saw the need for USO construction methods, and provided two methods~\cite{szabo2001usos}: 
Firstly, given an USO, all edges of a certain dimension can be \emph{flipped} at once, i.e., their orientation can be reversed.
Secondly, dimensions can be collapsed to create so-called \emph{inherited orientations}. 
These two constructions either preserve or reduce the dimension of the USO, and are thus not useful for creating large USOs.

It is also possible to perform more local modifications in an USO. Schurr and \szabo{} showed that under certain conditions, a subcube of an USO can be replaced by any other USO of the same dimension~\cite{schurr2004quadraticbound}. 
Unfortunately, there exist USOs in which no such modification is possible~\cite{schurr2004phd}. 
Schurr described an equivalence relation on edges of the same dimension, where flipping an equivalence class (also called a \emph{phase}) preserves the USO condition~\cite{schurr2004phd}. 
A Markov chain based on these flips converges to the uniform distribution, but its mixing rate remains unknown. 
Furthermore, the only known way of computing these equivalence classes is not efficient, as every pair of edges of the same dimension must be compared.

The most useful construction tool is the \emph{product construction} of Schurr and \szabo{}~\cite{schurr2004quadraticbound}, in which~$2^k$ USOs of the same dimension $d$ are connected by a $k$-dimensional \emph{frame} USO. 
This construction also admits two simple special cases ($k=1$ or $d=1$), as illustrated in \Cref{fig:productConstruction}.

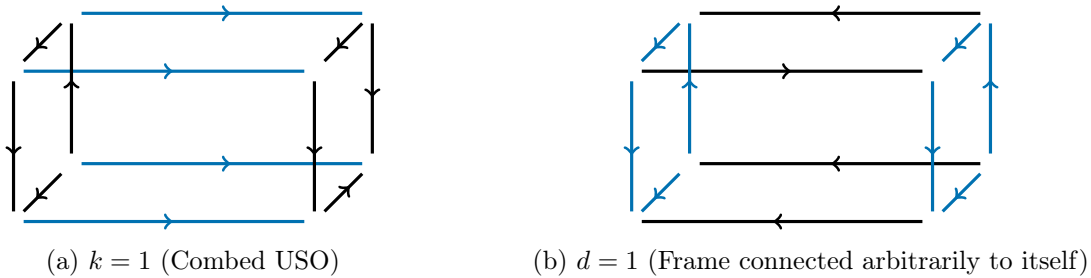
\begin{figure}[h!]
\centering
\begin{subfigure}{.49\textwidth}
  \centering
\begin{tikzpicture}[scale=2]
\foreach \x in {0,1}{
\foreach \y in {0,1}{
\foreach \z in {0,1}{
\node (\x\y\z) at (2*\x, \y, {1-\z}) {};
}}}
\draw[very thick,postaction=-<-, color=blue] (100) -- (000);
\draw[very thick,postaction=-<-, color=blue] (111) -- (011);
\draw[very thick,postaction=-<-, color=blue] (101) -- (001);
\draw[very thick,postaction=-<-, color=blue] (110) -- (010);
\draw[very thick,postaction=-<-] (000) -- (010);
\draw[very thick,postaction=-<-] (000) -- (001);
\draw[very thick,postaction=-<-] (011) -- (001);
\draw[very thick,postaction=-<-] (010) -- (011);
\draw[very thick,postaction=-<-] (101) -- (111);
\draw[very thick,postaction=-<-] (110) -- (111);
\draw[very thick,postaction=-<-] (100) -- (110);
\draw[very thick,postaction=-<-] (101) -- (100);
\end{tikzpicture}
  \caption{$k = 1$ (Combed USO)}
  \label{fig:k1}
\end{subfigure}
\begin{subfigure}{.49\textwidth}
  \centering
\begin{tikzpicture}[scale=2]
\foreach \x in {0,1}{
\foreach \y in {0,1}{
\foreach \z in {0,1}{
\node (\x\y\z) at (2*\x, \y, {1-\z}) {};
}}}
\draw[very thick,postaction=-<-] (000) -- (100);
\draw[very thick,postaction=-<-] (011) -- (111);
\draw[very thick,postaction=-<-] (001) -- (101);
\draw[very thick,postaction=-<-] (110) -- (010);
\draw[very thick,postaction=-<-, color=blue] (000) -- (010);
\draw[very thick,postaction=-<-, color=blue] (000) -- (001);
\draw[very thick,postaction=-<-, color=blue] (011) -- (001);
\draw[very thick,postaction=-<-, color=blue] (010) -- (011);
\draw[very thick,postaction=-<-, color=blue] (111) -- (101);
\draw[very thick,postaction=-<-, color=blue] (110) -- (111);
\draw[very thick,postaction=-<-, color=blue] (100) -- (110);
\draw[very thick,postaction=-<-, color=blue] (100) -- (101);
\end{tikzpicture}
  \caption{$d=1$ (Frame connected arbitrarily to itself)}
  \label{fig:d1}
\end{subfigure}
\caption{The two special cases of the product construction. The frame is colored in blue.}
\label{fig:productConstruction}
\end{figure}

While the product construction is a way to increase the dimension of USOs, it is still insufficient to construct many interesting USOs. 
For example, it is impossible to create an USO in which no edge can be flipped on its own (we say there are no \emph{flippable} edges), using product constructions with USOs that do contain flippable edges~\cite{schurr2004phd}.
It has also been shown that every USO created by successively applying product constructions is an easy input for the \textsc{Random Facet} sink-finding algorithm~\cite{weber2020randomfacet}, so this construction can not be used to prove a lower bound for this algorithm.

Since 2004, no construction methods applicable to all USOs have been discovered. 
Recent research has focused on constructions which are only applicable to USOs with certain properties, yielding again USOs with these properties. 
Such constructions have been found for \emph{bowed} USOs~\cite{weber2020randomfacet}, and for \emph{P-Cubes}~\cite{gao2020dcubes}.

\paragraph{Results.} Based on a remark of Schurr~\cite{schurr2004phd}, we prove a one-to-one correspondence between so-called \emph{$4\Integer^\dimensionK$-periodic tilings} and $k$-dimensional USOs. 
Representations of these tilings can be modified in the language of string rewriting, in particular this technique was used to disprove Keller's conjecture on unit cube tilings~\cite{keller1930conjecture,lagarias1992keller,mackey2002eightdimensional}.
We generalize these construction techniques and translate them into the language of USOs. 
Our generalization provides a very general framework with many parameters, and every choice of parameters is a new construction which can be applied to any USO. 
Given all $2$-dimensional USOs, repeated application of constructions from our framework can be used to generate all USOs of dimension $k\geq 1$, we thus call our framework \emph{universal}. 
Additionally, we show that we can realize the following constructions as special cases of our framework: (i) the dimension-reducing operations of taking facets, and taking inherited orientations, (ii) the dimension-preserving operations of flipping all edges of a dimension, and mirroring along a dimension, and (iii) the dimension-increasing product construction.
These constructions are the only constructions known to the authors which can be applied to all USOs. 
Finally, yet another special case of our construction gives a new dimension-preserving modification, the \emph{partial swap}, which can be applied to all USOs. We investigate the relationship between partial swaps and phases, and generalize the partial swap to the \emph{phase swap}.

\paragraph{Discussion.}
Unfortunately, our new construction exhibits a similar weakness to the phase flips of Schurr. 
While our construction is universal and each step in the universality proof is very systematic, our construction does not provide a suitable way to enumerate all USOs. 
This is because checking that the parameters of our construction fulfill the required conditions is a computationally expensive process.
The key benefit to our framework of string rewriting rules over previous constructions is the fact that each choice of the parameters induces a new construction applicable to all USOs. 
For every such construction, the effects of this transformation on the input USO are easy to understand and to describe.
Our construction may be useful in the future for proving algorithmic lower bounds. 
Furthermore, it may provide further insights into the structure of USOs.

\paragraph{Remarks.}
The bijection between $4\Integer^\dimensionK$-periodic tilings and USOs allows results from these fields to be exchanged.
In particular, Schurr \cite{schurr2004phd} computed the number of five-dimensional USOs, but was unsure whether his program yielded the actual number of five-dimensional USOs.
Independently, Mathew, Östergård and Popa \cite{mathew2013enumeratingtilings} explicitly generated the list of all isomorphy classes of $4\Integer^5$\nobreakdash-periodic tilings, arriving at the same total number of $4\Integer^5$-periodic tilings of 638\,560\,878\,292\,512.
This result confirms Schurr's count and provides an explicit list of all USOs in one dimension higher than previously accessible.

\paragraph{Paper Structure.}
In \Cref{sec:background}, we introduce the necessary notation and terminology concerning USOs and the previously known constructions in more detail. 
We also give an overview of unit cube tilings, discuss results related to Keller's conjecture, and prove the equivalence of $4\Integer^\dimensionK$-periodic tilings and USOs. 
In \Cref{sec:thebox}, we introduce string rewriting rules, introduce the partial swap operation and give some examples and intuition. 
In \Cref{sec:universality}, we prove the universality of our construction. 
In \Cref{sec:emulating}, we revisit the known constructions and analyze them in the framework of string rewriting rules, and introduce the phase swap operation. 
Finally, in \Cref{sec:conclusion} we ask some open questions.

\section{Background}\label{sec:background}
\subsection{Hypercubes and Orientations}
The \emph{$k$-dimensional hypercube graph} $Q_k$ (also called \emph{$k$-cube}) is the undirected graph on the vertex set \mbox{$V(Q_k) = \{0,1\}^k$}, where two vertices are connected by an edge if they differ in exactly one coordinate. 
For a vertex $v$ and a dimension $i\in [k]$, the vertex $v\xor i$ is the neighbor of $v$ which differs from $v$ in coordinate $i$. 
The edge between $v$ and $v\xor i$ is called an \emph{$i$-edge}, or an \emph{edge of dimension \(i\)}.
\begin{definition}
A \emph{face} of $Q_k$ described by a string $f\in\{0,1,*\}^k$ is the induced subgraph of $Q_k$ on the vertex set $V(f)\coloneqq \{v\in V(Q_k)\;|\;\forall i \in [k] : v_i=f_i \text{ or }f_i=* \}$.  
The \emph{dimension} of the face is the number of~$*$~symbols in $f$.
\end{definition}
\begin{definition}
A face of dimension $k-1$ is called a \emph{facet}. 
The facet described by the string with a~$1$ at the $i$-th position is called the \emph{upper $i$-facet}, and its opposite facet (described by the string with a $0$ at the $i$-th position) is called the \emph{lower $i$-facet}.
\end{definition}

An \emph{orientation} $O$ of $Q_k$ assigns each of the $k2^{k-1}$ edges a direction. 
We say that an edge $\{v,v\xor i\}$ is oriented \emph{upwards} if it is oriented from its endpoint in the lower $i$-facet to its endpoint in the upper $i$-facet. 
Otherwise, the edge is oriented \emph{downwards}. 
We can view  $O$ as a function $O:V(Q_k)\rightarrow \{0,1\}^k$, assigning each vertex a bitstring, such that $O(v)_i=1$ if and only if $\{v,v\xor i\}$ is oriented upwards. 
In our figures showing USOs, we indicate upwards edges with a yellow background. 

We note that our definition of orientations as functions differs from the ``outmap functions'' that are standard in USO literature~\cite{szabo2001usos}.
The benefit of this alternative definition will become apparent later.

The orientation where all edges are oriented downwards is called \emph{the canonical orientation}.
If all $i$-edges are oriented in the same direction for at least one dimension $i$, the orientation is said to be \emph{combed}.

We will now introduce some notation for modifying and combining orientations.

\begin{definition}
A \emph{partial orientation} on a vertex set $\sigma\subseteq V(Q_k)$ assigns an orientation only to the edges $E_\sigma$ of $Q_k$ which have at least one endpoint in $\sigma$. 
For an orientation $O$, $O_\sigma$ is the partial orientation on \(\sigma\) which agrees with $O$ on $\sigma$.
\end{definition}
\begin{definition}
Given a partial orientation $A$ on $\sigma$ and a partial orientation $B$ on $\sigma^C\coloneqq V(Q_k) \setminus \sigma$, such that $A$ and $B$ agree on the orientation of the edges across the cut $(\sigma,\sigma^C)$, the \emph{combined orientation} $O\coloneqq A\cup B$ is given by the function
\[ O(v)\coloneqq \begin{cases} 
      A(v) & \text{if }v\in\sigma, \\
      B(v) & \text{otherwise}.
   \end{cases}
\]

\end{definition}

\subsection{Unique Sink Orientations}
\begin{definition}
An orientation $O$ of $Q_k$ is a \emph{Unique Sink Orientation (USO)} if within each non-empty face $f$ of~$Q_k$, there exists exactly one vertex with no outgoing edges. That is, there is a unique sink in each face with respect to the orientation $O$.
\end{definition}

\szabo{} and Welzl~\cite{szabo2001usos} provide a useful characterization for USOs. 
The following theorem differs in appearance slightly from the one given in~\cite{szabo2001usos}, since we use different notation.

\begin{theorem}[\szabo{}-Welzl Condition]\label{thm:szabowelzl}
An orientation $O$ of $Q_k$ is an USO if and only if for all pairs of distinct vertices $v,w\in V(Q_k)$, we have
\[\exists i\in [k]:\; v_i\not=w_i \wedge \big(O(v)_i=O(w)_i\big).\]
\end{theorem}
This condition states that for every pair of vertices \(v,w\), within the minimal face containing both of them, there is a dimension $i$ in which the $i$-edges incident to $v$ and $w$ agree on their orientation.

All canonical orientations, where every edge of each dimension are oriented the same direction, are USOs, but there exist many more USOs. 
In fact, the number of USOs of $Q_k$ is $k^{\Theta(2^k)}$~\cite{matousek2006numberusos}, compared to $2^{k2^{k-1}}$ orientations in total.

\subsection{Overview of Previously Known Constructions}
\subsubsection{Product Constructions}
The product construction of Schurr and \szabo{}~\cite{schurr2004quadraticbound} is the most commonly used USO construction. 
Its inputs are a single $k$-dimensional USO $O$ (which we call the ``frame'') and a $d$-dimensional USO~$O_v$ for each vertex $v$ of \(O\). 
The vertices of \(O\) are replaced by the corresponding \(O_v\), and the edges of \(O\) are duplicated to fill the gaps.
This construction is described formally in the following Lemma.

\begin{lemma}[Product Construction {\cite[Lemma 3]{schurr2004quadraticbound}}]
Given an USO $O$ of $Q_k$ and USOs $O_v$ of $Q_d$ for each vertex $v\in V(Q_k)$, the product orientation $O'$ of $Q_{k+d}$ defined by
\[\forall v\in V(Q_{k+d}),i\in[k+d]: O'(v)_i\coloneqq \begin{cases}
    O(v_{1,\ldots,k})_i & \text{for }i\leq k,\\
    O_v(v_{k+1,\ldots,k+d})_{i-k} & \text{otherwise},
\end{cases}\]
is also an USO.
\end{lemma}

There are two particular special cases of this product construction we will emphasize.
The first case is when the frame is $1$-dimensional ($k=1$), and the second is when all USOs $O_v$ are $1$-dimensional ($d=1$).

If the frame is $1$-dimensional, the resulting $(d+1)$-dimensional product USO is composed of the two $d$-dimensional USOs $O_{(0)}$ and $O_{(1)}$ connected in a combed way. 

The other extreme is if the USOs $O_v$ are all $1$-dimensional; then the product construction replaces each vertex of the frame $O$ by a single edge. 
The resulting $(k+1)$-dimensional product USO consists of two copies of the $k$-dimensional frame $O$, with the orientation of the connecting edges decided by the~$O_v$.

\subsubsection{Inherited Orientations}
In contrast to all other known constructions, the inherited orientations of \szabo{} and Welzl~\cite{szabo2001usos} \emph{decrease} the dimension of the input USO, rather than increasing the dimension or leaving it fixed. 
The inputs of this construction are an USO $O$ of $Q_k$ and an integer $k'<k$. 
The inherited orientation~$O'$ of $O$ on $Q_{k'}$ is obtained by collapsing each $(k-k')$-dimensional face spanned by the dimensions~$[k]\setminus [k']$ to the sink of that face.
Formally, $O'$ is defined by
\[ \forall v\in V(Q_{k'}),i\in[k']: O'(v)_i\coloneqq  O(sink_O(v,\overbrace{*,\ldots, *}^{k-k'}))_i,\]
where $sink_O(f)$ is the sink of the face $f$ in $O$. 
The face $(v, *, \ldots, *)$ is exactly the face containing~$v$ that gets collapsed.
The paper which introduced this construction also showed that the result is always an USO~\cite[Lemma 3.1]{szabo2001usos}.

\subsubsection{Phase Flipping}

Schurr~\cite{schurr2004phd} described a way to determine the subsets of $i$-edges which can be flipped together.
The \emph{phases} of an USO are the minimal equivalence classes in the set of \(i\)-edges of the USO such that flipping any collection of these equivalence classes preserves the USO property.
The minimality condition ensures that flipping a set of \(i\)-edges which is not a union of phases will destroy the USO property.
It is possible that a phase consists of only a single edge $e$. In this case, we say that $e$ is \emph{flippable}.

\begin{lemma}[{\cite[Lemma 4.13]{schurr2004phd}}]
An edge $\{v,v\xor i\}$ in an USO $O$ is flippable if and only if $O$ fulfills \mbox{$O(v)=O(v\xor i)$}.
\end{lemma}
This corresponds to $v$ and $v\xor i$ having the same pattern of incident upwards and downwards edges.

The set of all $i$-edges is also a union of phases, as each edge is in some phase. 
Therefore, flipping all $i$-edges of an USO preserves the USO condition. This has already been observed by \szabo{} and Welzl~\cite{szabo2001usos}.

\subsubsection{Hypervertex Replacement}
As mentioned above, if the edges adjacent to both endpoints of some edge are oriented the same way in all dimensions, the edge can be flipped. 
The hypervertex replacement of Schurr and \szabo{}~\cite{schurr2004quadraticbound} generalizes this concept from single edges to higher-dimensional faces. 
If for each dimension $i$ which leaves some face, every $i$-edge incident to the face is oriented the same way, then the face can be replaced by any USO of the same dimension.

\begin{lemma}[{\cite[Corollary 6]{schurr2004quadraticbound}}]
For an USO $O$ of $Q_k$, and $f$ some $d$-dimensional face of $Q_k$ such that 
\[\forall v,w\in V(f): f_i\not=*\Longrightarrow O(v)_i=O(w)_i,\]
the orientation on the face $f$ can be replaced by another USO $O_f$ of $Q_d$. The resulting orientation~$O'$ defined by
\[\forall v\in V(Q_k),i\in [k]: O'(v)_i\coloneqq \begin{cases}
    O(v)_i & \text{if }f_i\not=* \text{ or } v\not\in f, \\
    O_f(\{v_j | f_j = *\})_i & \text{otherwise},
\end{cases}
\]
is also an USO.
\end{lemma}

\subsection{Unit Cube Tilings and Keller's Conjecture}
A \emph{tiling} of $\Reals^\dimensionK$ by unit cubes is an infinite collection of $\dimensionK$-dimensional axis-aligned unit cubes, such that (i) every point in $\Reals^\dimensionK$ is contained in at least one cube, and (ii) if a point is contained in multiple cubes, it lies on the boundary of all such cubes. 
In 1930, Keller conjectured that any such tiling must contain a pair of cubes (``twins''), such that their intersection is a facet of both cubes~\cite{keller1930conjecture}. 

\hajos{}~\cite{hajos1950factorisation} reformulated Keller's conjecture as a group theoretical statement. 
Based on this view, \szabo{}\footnote{Note that S{\'a}ndor \szabo{} working on tilings~\cite{szabo1993cubetilings} and Tibor \szabo{} working on USOs~\cite{szabo2001usos} are different people. Their work is independent.}~\cite{szabo1993cubetilings} showed that Keller's conjecture can be decided by looking only at so-called \emph{$4\Integer^\dimensionK$-periodic tilings}, i.e., there is a tiling without twins if and only if there is a $4\Integer^\dimensionK$-periodic tiling without twins.
In a $4\Integer^\dimensionK$-periodic tiling we tile the $k$-cube of side length $4$ by $2^\dimensionK$ integer-grid aligned $k$-cubes (\emph{tiles}) of side length $2$. 
These tiles may wrap around the boundary of the cube of side length $4$, exiting on one side and entering again on the opposite side (see \Cref{TilingsExamples}). More formally, each tile occupies the Minkowski sum of its bottom left corner and the cube $[0,2]^k$, component-wise modulo $4$.
This then defines a periodic tiling of $\Reals^\dimensionK$, as infinitely repeating the tiling of the cube of side length 4 fills~$\Reals^\dimensionK$.

\begin{example}
\label{TilingsExamples}
The following are four $4\Integer^2$-periodic tilings of $\Reals^2$. In the first and last tiling, all four pairs of adjacent tiles are twins.
\begin{center}
\begin{tikzpicture}[scale=0.8]
\drawTiling{0}{0}{0}{2}{2}{0}{2}{2}
\end{tikzpicture}
\begin{tikzpicture}[scale=0.8]
\drawTiling{0}{1}{0}{3}{2}{0}{2}{2}
\end{tikzpicture}
\begin{tikzpicture}[scale=0.8]
\drawTiling{1}{0}{0}{2}{3}{0}{2}{2}
\end{tikzpicture}
\begin{tikzpicture}[scale=0.8]
\drawTiling{1}{1}{1}{3}{3}{1}{3}{3}
\end{tikzpicture}
\end{center}
\end{example}

The set of coordinates of the lower left corners of the cubes in a $4\Integer^\dimensionK$-periodic tiling forms a set of $2^\dimensionK$ vectors in $\{0,1,2,3\}^\dimensionK$. \corradi{} and \szabo{}~\cite{corradi1990kellerconjecturestrings} described the conditions required for such a set to describe a valid $4\Integer^\dimensionK$-periodic tiling that is twin-free and thus a counterexample to Keller's conjecture.

In this model of sets of vectors describing $4\Integer^\dimensionK$-periodic tilings, various counterexamples for Keller's conjecture in higher dimensions were found. 
First, Lagarias and Shor discovered $12$- and $10$-dimensional counterexamples~\cite{lagarias1992keller}. 
Later, Mackey managed to improve on these ideas to achieve an $8$-dimensional counterexample~\cite{mackey2002eightdimensional}. 
While it was already known that Keller's conjecture is true for all $k\leq 6$ since 1940~\cite{perron1940kellersix}, the final case of $k=7$ has only been resolved recently when Brakensiek et al.~\cite{brakensiek2022kellerresolved} gave a computer-assisted proof that the conjecture holds for $k=7$.

In a short remark, Schurr~\cite{schurr2004phd} mentioned a one-to-one correspondence of \szabo{}'s $4\Integer^\dimensionK$-periodic tilings to USOs and that a tiling is a counterexample to Keller's conjecture if and only if the corresponding USO contains no flippable edges. 
This one-to-one correspondence was not further explained, so we provide the details here.

To describe valid $4\Integer^\dimensionK$-periodic tilings, Lagarias and Shor~\cite{lagarias1992keller} introduced the following graph:
\begin{definition}
\label{def:Gk}
For $k>0$, $G_k$ is the undirected graph with the vertex set $V(G_k) = \{0,1,2,3\}^k$.
Two vectors are adjacent if and only if in some coordinate their entries differ by exactly 2, i.e.,
\[ \{v,w\}\in E(G_\dimensionK) \Longleftrightarrow \exists i \in [\dimensionK]: \abs{v_i - w_i} = 2 . \]
\end{definition}

\begin{lemma}[Lagarias \& Shor~\cite{lagarias1992keller}]
\label{lem:Gk}
A set of vectors $\Clique \subseteq \{0,1,2,3\}^k$ describes the corners of the tiles of a $4\Integer^\dimensionK$-periodic cube tiling if and only if $\Clique$ induces a clique of size $2^\dimensionK$ in $G_\dimensionK$.
\end{lemma}
\Cref{lem:Gk} for tilings and \Cref{thm:szabowelzl} for USOs, respectively, are very similar: They both specify a condition that holds for all pairs of vectors (or vertices, respectively). The conditions both require the existence of a dimension with a certain property. We will now use this similarity to prove our desired bijection.

\begin{lemma}
\label{lem:uso=tiling}
There is a bijection between $4\Integer^\dimensionK$-periodic cube tilings and $k$-dimensional USOs.
\end{lemma}

\begin{proof}
We assume that $\Clique$ is a clique of size $2^\dimensionK$ in $G_\dimensionK$, and therefore corresponds to a $4\Integer^\dimensionK$-periodic cube tiling.
We can derive an orientation of the $\dimensionK$-dimensional hypercube $\Cube_\dimensionK$ in the following way.
We translate each vector~$v$ of $\Clique$ into a vertex of $\Cube_\dimensionK$ and its orientation by the two functions $F: \{0,1,2,3\}^\dimensionK \rightarrow V(\Cube_\dimensionK)$ and $\orientation: \{0,1,2,3\}^\dimensionK \rightarrow \{0, 1\}^\dimensionK$  with
\begin{align*}
    F(v)_i \coloneqq 
    \begin{cases}
    0 & \text{if } v_i = 0 \text{ or } v_i=1, \\
    1 & \text{if } v_i = 2 \text{ or } v_i=3, \\
    \end{cases}
\end{align*}
and 
\begin{align*}
    \orientation(v)_i \coloneqq 
    \begin{cases}
    0 & \text{if } v_i = 0 \text{ or }  v_i = 2, \\
    1 & \text{if } v_i = 1 \text{ or }  v_i = 3.\\
    \end{cases}
\end{align*}

We must prove that this indeed yields an USO.
For every two vectors $v, w \in \Clique$, it holds that they differ by exactly two in at least one element, say in position $i$. 
It follows that $F(v)_i \neq F(w)_i$ and the image of no two elements of $\Clique$ are the same vertex of $\Cube_\dimensionK$. 
Furthermore, it holds that $\orientation(v)_i = \orientation(w)_i$ and thus the condition of \Cref{thm:szabowelzl} is fulfilled for every two vertices.
This shows that the transformation of \(\Clique\) is an USO.

In the opposite direction, we must show that every USO $O$ on $Q_k$ corresponds to a clique of size $2^\dimensionK$ in $G_\dimensionK$.
Let $P: V(\Cube_\dimensionK)\times \{0,1\}^\dimensionK \rightarrow \{0, 1, 2, 3\}^\dimensionK$ be a function that maps vertices $v$ of the hypercube and their orientation $O(v)$ to a vector ($P$ is simply the inverse of $(F,O)$).
\begin{align*}
    P(v,o)_i \coloneqq 
    \begin{cases}
    0 & \text{if } v_i = 0 \text{ and } o_i = 0, \\
    1 & \text{if } v_i = 0 \text{ and } o_i = 1, \\
    2 & \text{if } v_i = 1 \text{ and } o_i = 0, \\
    3 & \text{if } v_i = 1 \text{ and } o_i = 1.
    \end{cases}
\end{align*}
For all pairs of vertices $v, w  \in \Cube_k$ with $v\neq w$, let $i$ be an index fulfilling the condition of \Cref{thm:szabowelzl}. Without loss of generality, we assume $v_i=0$ and thus $w_i=1$.
By \Cref{thm:szabowelzl}, the $i$-edges incident to $v$ and $w$ are both downwards or both upwards oriented. In the first case, $P(v,O(v))_i=0$ and $P(w,O(w))_i=2$. In the second case, $P(v,O(v))_i=1$ and $P(w,O(w))_i=3$. In both cases, $P(v,O(v))_i=P(w,O(w))_i-2$, and therefore $P(v,O(v))$ and $P(w,O(w))$ are adjacent in $G_\dimensionK$.

As the pair of functions $(F,O)$ and the function $P$ are bijections between $V(Q_k)\times \{0,1\}^k$ and $\{0,1,2,3\}^k$, they establish a bijection between $4\Integer^\dimensionK$-periodic tilings and USOs.
\end{proof}

In the following, we use the terms $4\Integer^\dimensionK$-periodic cube tiling and USO interchangeably and interpret a set of strings in $\{0,1,2,3\}^k$ as an USO according to the construction of the functions $F$ and~$O$ in the proof of \cref{lem:uso=tiling}.

\begin{example} The following USOs of $Q_2$ correspond to the tilings in \Cref{TilingsExamples}.

\begin{minipage}{0.95\textwidth}
\begin{center}
\begin{tikzpicture}[scale=0.8]
\drawUSO{0}{0}{0}{2}{2}{0}{2}{2}
\end{tikzpicture}\qquad
\begin{tikzpicture}[scale=0.8]
\drawUSO{0}{1}{0}{3}{2}{0}{2}{2}
\end{tikzpicture}\qquad
\begin{tikzpicture}[scale=0.8]
\drawUSO{1}{0}{0}{2}{3}{0}{2}{2}
\end{tikzpicture}\qquad
\begin{tikzpicture}[scale=0.8]
\drawUSO{1}{1}{1}{3}{3}{1}{3}{3}
\end{tikzpicture}
\end{center}
\end{minipage}
\end{example}

Two cubes $v,w\in\{0,1,2,3\}^\dimensionK$ in a $4\Integer^\dimensionK$-periodic tiling are twins if and only if they differ in only one coordinate. Note that in the USO this corresponds to $F(v)$ and $F(w)$ being neighbors and~$O(v)=O(w)$, thus the edge $\{F(v),F(w)\}$ being flippable.

\section{Rewriting Rules}\label{sec:thebox}

To disprove Keller's Conjecture, i.e., to construct an USO with no flippable edges, Lagarias and Shor~\cite{lagarias1992keller} use string rewriting to create higher dimensional tilings from lower dimensional tilings.
In this section, we generalize their technique to operations that can be applied to all USOs, so-called \emph{generalized rewriting rules}. Unlike Lagarias and Shor, we are also producing USOs with flippable edges.

We first define \emph{simple rewriting rules}, which rewrite a single digit in every string of an USO. For this, we need four lists, which specify what to replace each possible digit with. From Lagarias and Shor's approach we extract the conditions necessary to hold for these four lists, such that we can prove that the result is again an USO. We will later describe in \Cref{subsec:intuition} that in essence, a simple rewriting rule replaces each edge of some dimension $h$ by one of two $d$-dimensional USOs, depending on the orientation of the edge.

\begin{definition}
\label{def:rewritingRule}
	Let $S^{(0)}$, $S^{(1)}$, $S^{(2)},S^{(3)}\subseteq\{0,1,2,3\}^\dimension$ with the properties that
	\begin{itemize}
		\item[(i)] $\left(S^{(0)} \cup S^{(2)}\right)$ defines a $\dimension$-dimensional USO (a $4\Integer^\dimension$-periodic tiling) and $S^{(0)} \cap S^{(2)} = \emptyset$, and
		\item[(ii)] $\left(S^{(1)} \cup S^{(3)}\right)$ defines a $\dimension$-dimensional USO (a $4\Integer^\dimension$-periodic tiling) and $S^{(1)} \cap S^{(3)} = \emptyset$.
	\end{itemize}
	
	The sets $\left(S^{(0)},S^{(1)},S^{(2)},S^{(3)}\right)$ define a \emph{simple rewriting rule}. 
	We define the function $S_h$ to \emph{apply} this simple rewriting rule to a $k$-dimensional input USO $K$ on dimension $h\in [k]$.
	It maps subsets of $\{0,1,2,3\}^k$ to subsets of $\{0,1,2,3\}^{k+d-1}$.
	Applying the simple rewriting rule to a single vertex of the input USO is defined as follows:
	\[S_h(v) \coloneqq \left\{ v_{1}, \dots, v_{h-1}, s_1, \dots, s_d , v_{h+1}, \dots, v_\dimensionK \;|\; s\in S^{(v_h)}\right\}.\]
	We write $S_h(K)$ (for a set $K\subseteq \{0,1,2,3\}^\dimensionK$) for the union of the outputs of $S_h$ when applied to all elements of $K$, i.e., $S_h(K) \coloneqq \bigcup_{v\in K} S_h(v)$.
\end{definition}
For each string $v\in\Clique$, $S_h(v)$ produces a set of strings which depends on the value of the entry~$v_h$. 
For each element $s$ of~$S^{(v_h)}$, a string is generated by replacing the entry $v_h$ with $s$. 
The single vertex~$v$ is thus mapped to~$|S^{(v_h)}|$ vertices.

Note that some of the sets $S^{(m)}$ may be empty. 
In this case, when $|S^{(v_h)}|=0$, no strings are produced from $v$. 
While $S^{(0)}$ and $S^{(2)}$ are disjoint (and similarly $S^{(1)}$ and $S^{(3)}$), it is possible that the same string appears in both $(S^{(0)} \cup S^{(2)})$ and $(S^{(1)} \cup S^{(3)})$.

Lagarias and Shor applied their rewriting rules to all dimensions at once~\cite{lagarias1992keller}. 
As defined, our rewriting rules are only applied to a single dimension $h$ at once. 
To recreate Lagarias and Shor's process, we can apply our rewriting rules to each dimension in sequence.

\begin{example}
The following is a rewriting rule for $\dimension=2$.

\begin{minipage}{0.45\textwidth}
\begin{center}
\begin{tabular}{c|c|c|c}
$S^{(0)}$ & $S^{(1)}$ & $S^{(2)}$ & $S^{(3)}$\\ \hline
01 & 11 & 03 & 33 \\
& 31 & 20 & 13 \\
&    & 22 &    \\
\end{tabular}
\end{center}
\end{minipage}\hfill
\begin{minipage}{0.45\textwidth}
\begin{tikzpicture}[scale=1]
\drawUSO{0}{1}{0}{3}{2}{0}{2}{2}
\node at (1, 1) {$S^{(0)} \cup S^{(2)}$};
\end{tikzpicture}
\begin{tikzpicture}[scale=1]
\drawUSO{1}{1}{1}{3}{3}{1}{3}{3}
\node at (1, 1) {$S^{(1)} \cup S^{(3)}$};
\end{tikzpicture}
\end{minipage}

\noindent	
We apply this rewriting rule to dimension $h=1$ of the $2$-dimensional USO $\Clique = \{10, 30, 02, 22 \}$.
This means, we replace the first coordinate of each vertex.

\noindent
The result is $S_{1}(\Clique)= \{ 110, 310, 130, 330, 012, 032, 202, 222\}$.
\begin{figure}[h!]
\centering
\renewcommand{\vColorA}{red}
\renewcommand{\vColorB}{blue}
\renewcommand{\vColorC}{green}
\renewcommand{\vColorD}{orange}
\begin{tikzpicture}[scale=1]
\drawUSO{1}{0}{0}{2}{3}{0}{2}{2}
\path[thick,<->] (0,-0.5,0) edge node[below] {$h=1$} (2,-0.5, 0);
\node at (1, 1) {$K$};
\node at (4.5, 0.75) {$\Rightarrow$};
\node at (4.5, 1.25) {Apply rewriting rule};
\end{tikzpicture}
\begin{tikzpicture}
\renewcommand{\voffset}{0.5}
\node at (0, 0, 0) {};
\node (110) at (0, 0 + \voffset, 0) {\textcolor{\vColorA}{11}0};
\node (310) at (2, 0 + \voffset, 0) {\textcolor{\vColorA}{31}0};
\node (130) at (0, 2 + \voffset, 0) {\textcolor{\vColorC}{13}0};
\node (330) at (2, 2 + \voffset, 0) {\textcolor{\vColorC}{33}0};

\node (012) at (0, 0 + \voffset, -2) {\textcolor{\vColorB}{01}2};
\node (202) at (2, 0 + \voffset, -2) {\textcolor{\vColorD}{20}2};
\node (032) at (0, 2 + \voffset, -2) {\textcolor{\vColorD}{03}2};
\node (222) at (2, 2 + \voffset, -2) {\textcolor{\vColorD}{22}2};

\draw[very thick,postaction=->-] (110) -- (310);
\draw[very thick,postaction=->-] (110) -- (130);
\draw[very thick,postaction=-<-] (110) -- (012);

\draw[very thick,postaction=->-] (310) -- (330);
\draw[very thick,postaction=-<-] (310) -- (202);

\draw[very thick,postaction=-<-] (012) -- (202);
\draw[very thick,postaction=->-] (012) -- (032);

\draw[very thick,postaction=->-] (130) -- (330);
\draw[very thick,postaction=-<-] (130) -- (032);

\draw[very thick,postaction=-<-] (032) -- (222);
\draw[very thick,postaction=-<-] (330) -- (222);
\draw[very thick,postaction=-<-] (202) -- (222);

\end{tikzpicture}
\end{figure}
\end{example}
In the next lemma, we show that simple rewriting rules are correct USO constructions.

\begin{lemma}
\label{lem:rewritingResulsInUSO}
Applying a rewriting rule $S_h$ to an USO $\Clique$ of strings in $\{0,1,2,3\}^\dimensionK$ results in a valid USO $S_h(\Clique)$  of strings in $\{0,1,2,3\}^{\dimensionK+\dimension-1}$.
\end{lemma}
\begin{proof}
We must show that every two strings in $S_h(K)$ differ by exactly $2$ in some coordinate, and that $|S_h(K)|=2^{k+d-1}$. By \Cref{lem:Gk,lem:uso=tiling}, $S_h(K)$ thus describes a $k+d-1$-dimensional USO.
Without loss of generality, we assume $h=k$. 

Let $v\in K$ be some string. Every two strings in $S_h(v)$ differ by $2$ in some coordinate, since \Cref{def:rewritingRule} requires $S^{(v_h)}$ to be part of an USO.

Now, let $v,w\in K$ be two distinct strings. Since \(\Clique\) is an USO, \(v\) and \(w\) differ by $2$ in some coordinate. 
If they differ by \(2\) in coordinate $i\not=h$, every $v' \in S_h(v)$ also differs from every $w' \in S_h(w)$ by $2$ in that same coordinate $i$. 
Otherwise, $v$ and $w$ differ by $2$ in coordinate~$h$. 
In this case $v'$ differs from $w'$ by $2$ in some coordinate $i\in\{h,\ldots,h+d-1\}$, since \Cref{def:rewritingRule} requires that the disjoint unions $S^{(0)}\cup S^{(2)}$ and $S^{(1)}\cup S^{(3)}$ are USOs.

At last, we show that $S_h(K)$ contains exactly $2^{k+d-1}$ strings. Observe that we showed above that $S_h(v)\cap S_h(w)=\emptyset$. 
We thus have $|S_h(K)|=\sum_{v\in K}|S_h(v)|$. We can group the summands according to the $h$-edges of the USO. For the two endpoints $v,w$ of an $h$-edge, $v_h$ and $w_h$ must differ by exactly $2$. Thus, $\big(S^{(v_h)}\cup S^{(w_h)}\big)$ forms a $d$-dimensional USO, and therefore $S^{(v_h)}$ and $S^{(w_h)}$ together contain $2^d$ strings. As there are $2^{k-1}$ many $h$-edges, we get $2^{k-1}2^d=2^{k+d-1}$ strings in total.
\end{proof}

\subsection{Rewriting Rules for \texorpdfstring{$d=1$}{d=1}}

Before we describe the full generality of rewriting rules, we give various rewriting rules with $d=1$ to build intuition for string rewriting operations.

\subsubsection{The Identity}
\begin{minipage}{0.3\textwidth}
\begin{center}
\begin{tabular}{c|c|c|c}
$S^{(0)}$ & $S^{(1)}$ & $S^{(2)}$ & $S^{(3)}$\\ \hline
0 & 1 & 2 & 3 \\
\end{tabular}
\end{center}
\end{minipage}\hfill
\begin{minipage}{0.7\textwidth}
This rewriting rule changes nothing, applying it to an USO results in the same USO again.
\end{minipage}

\subsubsection{Combing}
\begin{minipage}{0.3\textwidth}
\begin{center}
\begin{tabular}{c|c|c|c}
$S^{(0)}$ & $S^{(1)}$ & $S^{(2)}$ & $S^{(3)}$\\ \hline
0 & 0 & 2 & 2 \\ 
\end{tabular}
\end{center}
\end{minipage}\hfill
\begin{minipage}{0.7\textwidth}
This rewriting rules combs dimension $h$, making all $h$-edges oriented in the same direction (downwards).
\end{minipage}

\subsubsection{Flipping All Edges in One Dimension}
\label{sec:flipping}
\begin{minipage}{0.3\textwidth}
\begin{center}
\begin{tabular}{c|c|c|c}
$S^{(0)}$ & $S^{(1)}$ & $S^{(2)}$ & $S^{(3)}$\\ \hline
1 & 0 & 3 & 2 \\
\end{tabular}
\end{center}
\end{minipage}\hfill
\begin{minipage}{0.7\textwidth}
This rewriting rule replaces all upwards with downwards edges and all downwards with upwards edges. Applied to dimension $h$, this rule flips all $h$-edges (as in \cite[Lemma 2.1]{szabo2001usos}).
\end{minipage}\vskip\baselineskip

Until this point, the example rewriting rules mapped each vertex to the same vertex, and only possibly changed the orientation of the incident $h$-edge. This does not have to be the case, as we see in the remaining rewriting rules.
\subsubsection{Copying Upper \texorpdfstring{$h$}{h}-Facet}
\begin{minipage}{0.3\textwidth}
\begin{center}
\begin{tabular}{c|c|c|c}
$S^{(0)}$ & $S^{(1)}$ & $S^{(2)}$ & $S^{(3)}$\\ \hline
 &  & 0 & 1 \\
 &  & 2 & 3 \\
\end{tabular}
\end{center}
\end{minipage}\hfill
\begin{minipage}{0.7\textwidth}
Applied to dimension $h$, this rewriting rule copies the upper $h$-facet of the input USO into the lower $h$-facet.
The directions of the $h$-edges are unchanged. The lower facet can also be copied into the upper facet by a symmetrical rewriting rule.
\end{minipage}

\subsubsection{Swapping \texorpdfstring{$h$}{h}-Facets}
\label{sec:swapping}
\begin{minipage}{0.3\textwidth}
\begin{center}
\begin{tabular}{c|c|c|c}
$S^{(0)}$ & $S^{(1)}$ & $S^{(2)}$ & $S^{(3)}$\\ \hline
2 & 3 & 0 & 1 \\
\end{tabular}
\end{center}
\end{minipage}\hfill
\begin{minipage}{0.7\textwidth}
Applied to dimension $h$, this rewriting rule swaps the upper $h$-facet of the input USO and the lower $h$-facet.
The directions of the $h$-edges are unchanged.
\end{minipage}

\subsubsection{Partial Swap}
\label{sec:nonCubicSubgraphs}
\begin{minipage}{0.3\textwidth}
\begin{center}
\begin{tabular}{c|c|c|c}
$S^{(0)}$ & $S^{(1)}$ & $S^{(2)}$ & $S^{(3)}$\\ \hline
0 & 3 & 2 & 1 \\ 
\end{tabular}
\end{center}
\end{minipage}\hfill
\begin{minipage}{0.7\textwidth}
This rewriting rule swaps the induced subgraphs of the lower and upper $h$-facets of vertices incident to upwards $h$-edges.
\end{minipage}\vskip\baselineskip

This last operation had not been previously documented in the literature. This is a new USO-preserving operation, and we formalize it in the following theorem.

\begin{theorem}
\label{thm:swappingNonCubicSubgraphs}
Let $\Clique$ be a $k$-dimensional USO and $h\in [k]$ some dimension. 
Let $U\subset V(K)$ be the set of vertices in the upper $h$-facet of $\Clique$ that are incident to an upwards $h$-edge. 
Similarly, let $L\subset V(K)$ be the neighbors of $U$ in the lower $h$-facet. 
Now let $\Clique_U$ be the subgraph induced by $U$ and $\Clique_L$ be the subgraph induced by $L$. 
Note that these graphs might have more than one connected component.
Swapping these two subgraphs yields another orientation $\Clique'$, which is an USO.
\end{theorem}
\begin{proof}
Applying the rewriting rule $(S^{(0)}=\{0\}, S^{(1)}=\{3\}, S^{(2)}=\{2\},S^{(3)}=\{1\})$ to dimension~$h$ of $\Clique$ yields $\Clique'$, since vertices incident to downwards $h$-edges are left unchanged and vertices incident to upwards $h$-edges are moved to the opposite $h$-facets. This rewriting rule fulfills all conditions of \cref{def:rewritingRule}, and thus the result of applying it to an USO must be an USO again by \cref{lem:rewritingResulsInUSO}.
\end{proof}

\begin{example}
We apply the rewriting rule $(S^{(0)}=\{0\}, S^{(1)}=\{3\}, S^{(2)}=\{2\},S^{(3)}=\{1\})$ to the USO $K = \{ 110, 310,  012, 202, 031, 230, 033, 222 \}$ in dimension $h=2$, i.e., we rewrite the second coordinate of each vertex. In the resulting USO, the subgraphs $K_L$ and $K_U$ swapped places.

\renewcommand{\vColorA}{red}
\renewcommand{\vColorB}{blue}
\newcommand{\offset}{6}
\newcommand{\opacity}{0.15}
\begin{center}
\begin{tikzpicture}[scale=0.85]
\node at (-1.3,0,0) {};
\node at (1.3+2.3+\offset+\offset,2.3,-2.3) {};
\path[thick,<->] (-0.7,0,0) edge node[left] {$h$} (-0.7,2.3,0);

\node (110a) at (0, 0, 0) {110};
\node (310a) at (2.3, 0, 0) {310};
\node (031a) at (0, 2.3, 0) {031};
\node (230a) at (2.3, 2.3, 0) {230};

\node (012a) at (0, 0, -2.3) {012};
\node (202a) at (2.3, 0, -2.3) {202};
\node (033a) at (0, 2.3, -2.3) {033};
\node (222a) at (2.3, 2.3, -2.3) {222};

\draw[very thick,postaction=->-] (110a) -- (310a);
\draw[very thick,postaction=->-] (110a) -- (031a);
\draw[very thick,postaction=-<-] (110a) -- (012a);

\draw[very thick,postaction=->-] (310a) -- (230a);
\draw[very thick,postaction=-<-] (310a) -- (202a);

\draw[very thick,postaction=-<-] (012a) -- (202a);
\draw[very thick,postaction=->-] (012a) -- (033a);

\draw[very thick,postaction=-<-] (031a) -- (230a);
\draw[very thick,postaction=->-] (031a) -- (033a);

\draw[very thick,postaction=-<-] (033a) -- (222a);
\draw[very thick,postaction=-<-] (230a) -- (222a);
\draw[very thick,postaction=-<-] (202a) -- (222a);

\draw[\vColorA, line width=17pt, line cap = round, opacity=\opacity, rounded corners = 0.05pt] (0,0, -2.5) -- (0, 0, 0) -- (2.5, 0, 0);
\draw[\vColorB, line width=17pt, line cap = round, opacity=\opacity, rounded corners = 0.05pt] (0,2.3, -2.5)--(0, 2.3, 0) -- (2.5, 2.3, 0);

\node at (3.5,0,0) {\textcolor{\vColorA}{$\Clique_L$}};
\node at (-1,2.5,-2) {\textcolor{\vColorB}{$\Clique_U$}};

\node at (4.5, 1,0) {$\Rightarrow$};
\node at (4.5, 2, 0) {\setlength\extrarowheight{-3pt}\begin{tabular}{c} Apply \\ rewriting \\ rule \end{tabular}};

\node (110b) at (0+\offset, 0, 0) {130};
\node (310b) at (2.3+\offset, 0, 0) {330};
\node (031b) at (0+\offset, 2.3, 0) {011};
\node (230b) at (2.3+\offset, 2.3, 0) {210};

\node (012b) at (0+\offset, 0, -2.3) {{032}};
\node (202b) at (2.3+\offset, 0, -2.3) {202};
\node (033b) at (0+\offset, 2.3, -2.3) {013};
\node (222b) at (2.3+\offset, 2.3, -2.3) {222};

\draw (310b) -- (110b);
\draw (110b) -- (031b);
\draw (110b) -- (012b);

\draw (310b) -- (230b);
\draw (310b) -- (202b);

\draw (012b) -- (202b);
\draw (012b) -- (033b);

\draw (031b) -- (230b);
\draw (033b) -- (031b);

\draw (033b) -- (222b);
\draw (230b) -- (222b);
\draw (202b) -- (222b);

\draw[\vColorA, line width=17pt, line cap = round, opacity=\opacity, rounded corners = 0.05pt] (0+\offset,0, -2.5) -- (0+\offset, 0, 0) -- (2.5+\offset, 0, 0);
\draw[\vColorB, line width=17pt, line cap = round, opacity=\opacity, rounded corners = 0.05pt] (0+\offset,2.3, -2.5)--(0+\offset, 2.3, 0) -- (2.5+\offset, 2.3, 0);

\node at (4.5+\offset, 1,0) {$\Rightarrow$};
\node at (4.5+\offset, 2, 0) {\setlength\extrarowheight{-3pt}\begin{tabular}{c} Fix \\ vertex \\ position \end{tabular}};

\node (011c) at (0+\offset+\offset, 0, 0) {011};
\node (210c) at (2.3+\offset+\offset, 0, 0) {210};
\node (130c) at (0+\offset+\offset, 2.3, 0) {130};
\node (330c) at (2.3+\offset+\offset, 2.3, 0) {330};

\node (013c) at (0+\offset+\offset, 0, -2.3) {013};
\node (202c) at (2.3+\offset+\offset, 0, -2.3) {202};
\node (032c) at (0+\offset+\offset, 2.3, -2.3) {032};
\node (222c) at (2.3+\offset+\offset, 2.3, -2.3) {222};

\draw[very thick,postaction=-<-] (011c) -- (210c);
\draw[very thick,postaction=->-] (011c) -- (130c);
\draw[very thick,postaction=->-] (011c) -- (013c);

\draw[very thick,postaction=->-] (210c) -- (330c);
\draw[very thick,postaction=-<-] (210c) -- (202c);

\draw[very thick,postaction=-<-] (013c) -- (202c);
\draw[very thick,postaction=->-] (013c) -- (032c);

\draw[very thick,postaction=->-] (130c) -- (330c);
\draw[very thick,postaction=-<-] (130c) -- (032c);

\draw[very thick,postaction=-<-] (032c) -- (222c);
\draw[very thick,postaction=-<-] (330c) -- (222c);
\draw[very thick,postaction=-<-] (202c) -- (222c);

\draw[\vColorB, line width=17pt, line cap = round, opacity=\opacity, rounded corners = 0.05pt] (0+\offset+\offset,0, -2.5) -- (0+\offset+\offset, 0, 0) -- (2.5+\offset+\offset, 0, 0);
\draw[\vColorA, line width=17pt, line cap = round, opacity=\opacity, rounded corners = 0.05pt] (0+\offset+\offset,2.3, -2.5)--(0+\offset+\offset, 2.3, 0) -- (2.5+\offset+\offset, 2.3, 0);

\node at (3.5+\offset+\offset,0,0) {\textcolor{\vColorB}{$\Clique_U$}};
\node at (-1+\offset+\offset,2.5,-2) {\textcolor{\vColorA}{$\Clique_L$}};

\end{tikzpicture}
\end{center}


\end{example}

\subsection{Generalized Rewriting Rules}

To arrive at their counterexamples to Keller's conjecture, Lagarias and Shor require a more general rewriting technique~\cite{lagarias1992keller}. In fact, we can prove that simple rewriting rules are not able to generate USOs without flippable edges only from USOs with flippable edges. Simple rewriting rules are thus not expressive enough to be universal.

\begin{lemma}
Applying a simple rewriting rule in which $\left(S^{(0)}\cup S^{(2)}\right)$ and $\left(S^{(1)}\cup S^{(3)}\right)$ contain at least one flippable edge to an USO $K$ which contains at least one flippable edge yields an USO that also contains at least one flippable edge.
\end{lemma}
\begin{proof}
Let $S_h$ be a function applying such a rewriting rule in dimension $h$. Let $K$ be the input USO, containing a flippable edge $\{v,v\xor i\}$. All the edges incident to $v$ and $w:=v\xor i$ therefore have the same orientation. We distinguish two cases, depending on $h$.

First, assume $h\not=i$. It must hold that $v_h=w_h$, and thus the same set $S^{(v_h)}=S^{(w_h)}$ is used to rewrite both of these strings.
Let $v'\in S_h(v)$ and let $w'\in S_h(w)$ be the corresponding string, obtained by using the same $s\in S^{(v_h)}$. Then, $v'$ and $w'$
form a flippable edge $\{v',w'\}$.

Second, assume $h=i$. Then, $S_h(\{v,v\xor i\})$ is a hypervertex in $S_h(K)$, and equal to one of the two USOs $\left(S^{(0)}\cup S^{(2)}\right)$ or $\left(S^{(1)}\cup S^{(3)}\right)$. Since both of these USOs are assumed to contain a flippable edge, so must $S_h(K)$.
\end{proof}

To circumvent this issue, Lagarias and Shor do not only use the four digits ${0,1,2,3}$ in their input clique, but also $0'$ and $1'$. We generalize our construction based on this idea, by letting the input USO specify one of $i$ labels at each vertex. 

For the full generality of our rewriting framework, the sets $S^{(m)}$ are replaced by a list of $i$ sets~$S^{(m)}_{1,\ldots,i}$ each, where the indices correspond to the possible labels attached to the vertices of the input USO. All the compatibility requirements are appropriately expanded.

\begin{definition}\label{def:generalizedRewritingRule}
Let $i \in \mathbb{N}$, \(S^{(0)}_{1, \ldots, i}, S^{(1)}_{1,\ldots,i}, S^{(2)}_{1, \ldots, i}, S^{(3)}_{1,\ldots i} \subseteq \{0,1,2,3\}^\dimension\) with the properties that
\begin{itemize}
		\item[(i)] $\left(S^{(0)}_{j} \cup S^{(2)}_{j'}\right)$ defines a $\dimension$-dimensional USO and $S^{(0)}_{j} \cap S^{(2)}_{j'} = \emptyset$ for all pairs \(j,j' \in [i]\), and
		\item[(ii)] $\left(S^{(1)}_{j} \cup S^{(3)}_{j'}\right)$ defines a $\dimension$-dimensional USO and $S^{(1)}_{j} \cap S^{(3)}_{j'} = \emptyset$ for all pairs \(j,j' \in [i]\).
	\end{itemize}

The sets \(\left(S^{(0)}_{1, \ldots, i}, S^{(1)}_{1,\ldots,i}, S^{(2)}_{1, \ldots, i}, S^{(3)}_{1,\ldots i}\right)\) define a \emph{generalized rewriting rule}. 
We define the function~$T_h$ to \emph{apply} this generalized rewriting rule to a $k$-dimensional input USO $K$ on dimension~$h\in [k]$.
It maps subsets of \(\{0,1,2,3\}^k \times [i]\) to subsets of $\{0,1,2,3\}^{k+d-1}$.
Applying the generalized rewriting rule to a single vertex $v$ labelled $j$ of the input USO is defined as follows:
\begin{align*}
T_h (v,j) \coloneqq \left\{v_1, \ldots, v_{h-1}, t_1, \ldots, t_d, v_{h+1}, \ldots, v_{k} \mid t \in S^{(v_h)}_{j}\right\}.    
\end{align*}
Similar to \Cref{def:rewritingRule} we extend the function $T_h$ from single inputs to sets: For a set \mbox{$K\subseteq\{0,1,2,3\}^\dimensionK$} and a labelling function $L:K\to [i]$, $T_h(K,L)\coloneqq \bigcup_{v\in K}T_h(v,L(v))$.
\end{definition}

Note that we can use duplicate sets $S^{(m)}_j=S^{(m)}_{j'}$ in case we want to have fewer than $i$ sets for~$m \in \{0,1,2,3\}$. We next state a version of \Cref{lem:rewritingResulsInUSO} for generalized rewriting rules.
\begin{lemma}
Let \(K\) be an USO of strings in $\{0,1,2,3\}^\dimensionK$, and \mbox{\(L : K \to [i]\)} an additional labelling function. Then \(T_h(K,L)\) is an USO of strings in $\{0,1,2,3\}^{\dimensionK+\dimension-1}$.
\end{lemma}
\begin{proof}
The proof of \Cref{lem:rewritingResulsInUSO} applies \textit{mutatis mutandis} to this lemma, with the additional observation that \Cref{def:generalizedRewritingRule} guarantees the disjointness and coherence conditions required.
\end{proof}

Since every set $S_j^{(0)}$ must ``match'' every set $S_{j'}^{(2)}$, all the sets $S_j^{(0)}$ replacing the $0$ entries must have the same size, and must describe the same part of the cube. We formalize this intuition in the following observation.

\begin{observation}
\label{obs:sameUnorientedVertices}
For all $j, j' \in [i]$ and $m \in \{0,1,2,3\}$, the sets $S^{(m)}_j$ and $S^{(m)}_{j'}$ always contain the same vertices of the unoriented cube, i.e., $\left\{ F(s) \mid s \in S^{(m)}_j\right\} = \left\{ F(s') \mid  s' \in S^{(m)}_{j'}\right\}$ for~$F$ defined as in the proof of \cref{lem:uso=tiling}.
\end{observation}

\subsection{(Generalized) Rewriting Rules From the USO Viewpoint}\label{subsec:intuition}

In this section, we give some intuition about the effect of generalized rewriting rules on an USO.
Of course the space of rewriting rules is very large, but nevertheless, there are some general statements that can be made.
We focus on the \(2\)-faces of the USO containing $h$-edges, which illustrate the constraints and effects of the rewriting.

Let \(T_h\) be a generalized rewriting rule.
Let \(K\) be the input USO labelled by the function $L$, and let the vertices \(v_1, v_2, w_1, w_2\) be a \(2\)-face $f$ of \(K\), where \(\{v_1,v_2\}\) and \(\{w_1,w_2\}\) are $h$-edges.
The rewriting rule replaces those two $h$-edges by the $d$-dimensional USOs $T_h(\{v_1, v_2 \}, L)$ and $T_h(\{w_1, w_2 \}, L)$.
Instead of the edges $\{v_1, w_1\}$ and $\{v_2, w_2\}$, there are now $2^d$ new edges between  $T_h(\{v_1, v_2 \}, L)$ and~$T_h(\{w_1, w_2 \}, L)$ as can be seen in \Cref{fig:sketchOfRewritingRule}.
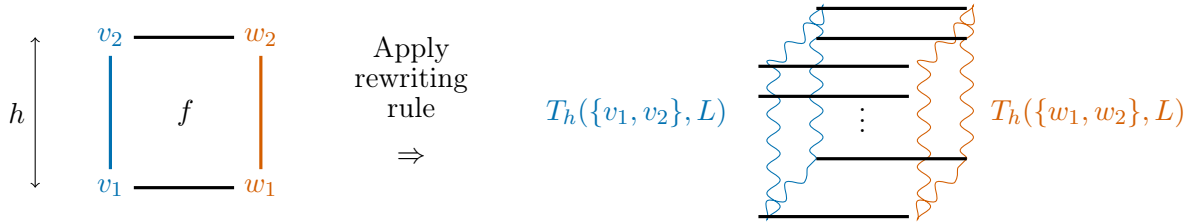
\begin{figure}[h!]
\centering
\begin{tikzpicture}[scale=2]
\newcommand{\myVAlign}{0}
\newcommand{\offset}{4.5}

\node(00) at (0, \myVAlign, 0.5) {\textcolor{blue}{$v_1$}};
\node(01) at (0, 1 +\myVAlign, 0.5) {\textcolor{blue}{$v_2$}};
\node(10) at (1, \myVAlign, 0.5) {\textcolor{red}{$w_1$}};
\node(11) at (1, 1+ \myVAlign, 0.5) {\textcolor{red}{$w_2$}};
\node at (0.5, 0.5 +\myVAlign, 0.5) {$f$};
\draw[very thick] (00) -- (10);
\draw[very thick, blue] (00) -- (01);
\draw[very thick, red] (10) -- (11);
\draw[very thick] (11) -- (01);
\draw[<->] (-0.5, \myVAlign, 0.5) --node[left] {$h$} (-0.5, 1+\myVAlign, 0.5);

\node at (2, 0.2+\myVAlign, 0.5) {$\Rightarrow$};
\node at (2, 0.7+\myVAlign,0.5) {\setlength\extrarowheight{-3pt}\begin{tabular}{c} Apply \\ rewriting \\ rule \end{tabular}};

\draw[color=blue,decorate,decoration=snake] (\offset, 0, 0) -- (\offset,1,0) -- (\offset,1,1) -- (\offset,0,1) -- (\offset,0,0);
\node at (-1+\offset, 0.5, 0.5) {\textcolor{blue}{$T_h(\{v_1, v_2 \}, L)$}};
\draw[color=red,decorate,decoration=snake] (1+\offset, 0, 0) -- (1+\offset,1,0) -- (1+\offset,1,1) -- (1+\offset,0,1) -- (1+\offset,0,0);
\node at (2+\offset, 0.5, 0.5) {\textcolor{red}{$T_h(\{w_1, w_2 \}, L)$}};

\node at (0.5+\offset, 0.5, 0.5) {$\vdots$};

\draw[very thick] (0+\offset,0,0) -- (1+\offset,0,0);
\draw[very thick] (0+\offset,0.8,0) -- (1+\offset,0.8,0);
\draw[very thick] (0+\offset,0.8,1) -- (1+\offset,0.8,1);
\draw[very thick] (0+\offset,1,0) -- (1+\offset,1,0);
\draw[very thick] (0+\offset,1,1) -- (1+\offset,1,1);
\draw[very thick] (0+\offset,0,1) -- (1+\offset,0,1);
\end{tikzpicture}
\caption{Sketch of the effect of the generalized rewriting rule on the $2$-face $f$.}
\label{fig:sketchOfRewritingRule}
\end{figure}

It holds that if $(v_1, v_2)$ is a downwards edge, $T_h(\{v_1, v_2 \}, L) = \left(S^{(0)}_{L(v_1)} \cup S^{(2)}_{L(v_2)}\right)$. 
If $(v_1, v_2)$ is an upwards edge,  $T_h(\{v_1, v_2 \}, L) = \left(S^{(1)}_{L(v_1)} \cup S^{(3)}_{L(v_2)}\right)$.
Analogously, the edge $(w_1, w_2)$ is replaced by the respective union of sets. 
In either case, this is guaranteed to be an USO by the conditions (i) and (ii) in Definition \ref{def:generalizedRewritingRule}.
The remaining information that requires clarification is the orientation of the edges between these USOs.

\paragraph{Case 1}
The $h$-edges have opposing orientations.

\begin{minipage}{0.15\textwidth}
\begin{tikzpicture}
\node(00) at (0, 0) {$v_1$};
\node(01) at (0, 1) {$v_2$};
\node(10) at (1, 0) {$w_1$};
\node(11) at (1, 1) {$w_2$};
\node at (0.5, 0.5) {$f$};
\draw[very thick] (00) -- (10);
\draw[very thick,postaction=-<-] (00) -- (01);
\draw[very thick,postaction=->-] (10) -- (11);
\draw[very thick] (11) -- (01);
\draw[<->] (-0.5, 0) --node[left] {$h$} (-0.5, 1);
\end{tikzpicture}
\end{minipage}\hfill
\begin{minipage}{0.8\textwidth}
\vspace{0.1em}
In this case $\{v_1, w_1\}$ and $\{v_2, w_2\}$ must have the same orientation, as $K$ would not be USO otherwise.
Thus, all connecting edges between $T_h(\{v_1, v_2 \}, L)$ and~$T_h(\{w_1, w_2 \}, L)$ are connected in a combed way according to the orientation of~$\{v_1, w_1\}$ and $\{v_2, w_2\}$. 
\end{minipage}

\paragraph{Case 2}
Both $h$-edges have the same orientation.

In this case $\{v_1, w_1\}$ and $\{v_2, w_2\}$ can (but not necessarily do) have opposing orientations.
Recalling \Cref{obs:sameUnorientedVertices}, note that all vertices in $T_h(v_1, L(v_1))$ are only connected to vertices in~$T_h(w_1, L(w_1))$, but not to vertices in~$T_h(w_2, L(w_2))$. Analogously, vertices in $T_h(v_2, L(v_2))$ are only connected to $T_h(w_2, L(w_2))$.

\begin{minipage}{0.15\textwidth}
\begin{tikzpicture}
\node(00) at (0, 0) {$v_1$};
\node(01) at (0, 1) {$v_2$};
\node(10) at (1, 0) {$w_1$};
\node(11) at (1, 1) {$w_2$};
\node at (0.5, 0.5) {$f$};
\draw[very thick] (00) -- (10);
\draw[very thick,postaction=-<-] (00) -- (01);
\draw[very thick,postaction=-<-] (10) -- (11);
\draw[very thick] (11) -- (01);
\draw[<->] (-0.5, 0) --node[left] {$h$} (-0.5, 1);
\end{tikzpicture}
\end{minipage}\hfill
\begin{minipage}{0.8\textwidth}
The edges between vertices from $T_h(v_1, L(v_1))$ and $T_h(w_1, L(w_1))$ are oriented the same way as $\{v_1, w_1\}$.
The edges between vertices from $T_h(v_2, L(v_2))$ and~$T_h(w_2, L(w_2))$ are oriented the same way as $\{v_2, w_2\}$.
\end{minipage}

\paragraph{}We summarize: In the resulting USO, any face spanned by the dimensions $\{h,\ldots, h+d-1\}$ is one of the USOs made up from two sets of the generalized rewriting rule. Furthermore, between any two such USOs the edges are either combed, or are split into upwards and downwards edges according to the split of the $\big(S^{(0)}_j\cup S^{(2)}_{j'} \big)$ into their parts, or according to the split of the $\big(S^{(0)}_j\cup S^{(2)}_{j'}\big)$ into their parts.


\section{Universality of the Construction}\label{sec:universality}

Our construction is universal, meaning it is sufficiently general to generate all USOs, using only the 1-dimensional USOs, and the $2$-dimensional ``bow'' as base cases.

\begin{theorem}[Universality]
\label{thm:repeatedApplication}
Starting with the set of both $1$-dimensional USOs $\{0,2\}$ and $\{1,3\}$, one can generate every USO of dimension $n\geq 1$ by repeated application of generalized rewriting rules to the bow $\{01,20,03,22\}$, where the sets used in each rewriting rule are partial orientations from the set of already obtained USOs.
\end{theorem}

To prove this theorem, we show that for any $n$-dimensional USO, there exists a generalized rewriting rule which creates this USO by only using $(n-1)$-dimensional USOs and the bow.
From this, \Cref{thm:repeatedApplication} follows as a direct consequence.

\begin{lemma}\label{thm:universality}
Let $K$ be an $n$-dimensional USO. Then there exists a generalized rewriting rule \(\left(S^{(0)}_{1,2}, S^{(1)}_{1,2}, S^{(2)}_{1,2}, S^{(3)}_{1,2}\right)\), where each $S^{(m)}_j$ is a (partial) \((n-1)\)-dimensional USO, and
\[K=T_1(bow=\{01,20,03,22\},L), \text{ for } L(01)=2,L(03)=1,L(20)=2,L(22)=1.\]
\end{lemma}

\begin{proof}
As our input USO has no upwards edges in direction $h=1$, the sets \(S^{(1)}_{1,\ldots,i}\) and \(S^{(3)}_{1,\ldots,i}\) are not used at all in this construction.
We simply define all \(S^{(1)}_j\) to be the empty set and all \(S^{(3)}_j\) to be the canonical $(n-1)$-dimensional USO.

We use \(i=2\) many labels, and thus need to define the four sets \(S^{(0)}_1,S^{(0)}_2,S^{(2)}_1,S^{(2)}_2\). Note that when applying the rewriting rule to the bow labelled by $L$, each of these sets is used to rewrite exactly one string. Furthermore, each vertex of the bow has a different value in its second coordinate, which will not be changed by the rewriting rule. We thus pick our four sets as follows:
\begin{alignat*}{1}
    S^{(0)}_1\coloneqq V_3, \qquad
    S^{(2)}_1\coloneqq V_2, \qquad
    S^{(0)}_2\coloneqq V_1, \qquad
    S^{(2)}_2\coloneqq V_0,
\end{alignat*}
where $V_m\coloneqq \{v_1,\ldots,v_{n-1} \;|\; v\in K \text{ and } v_n=m\}$ is the set of prefixes of strings in $K$ ending with~$m$.

It is now easy to see that applying this rewriting rule to the bow labelled by $L$ results in $K$ again, as the rewriting of each vertex of the bow results in the set of strings of $K$ ending in one of four values.

It remains to prove that these sets form a valid generalized rewriting rule. First, note that the disjoint union of $S_1^{(0)}$ and $S_1^{(2)}$ forms the upper $n$-facet of the target USO $K$. Similarly, the disjoint union of $S_2^{(0)}$ and $S_2^{(2)}$ forms the lower $n$-facet of $K$. The partitioning of these facets into their parts depends on the direction of the edges connecting these facets in $K$, with vertices incident to upwards $n$-edges lying in $S_1^{(0)}$ or $S_2^{(0)}$.

To show correctness of the generalized rewriting rule, we only need to check that four combinations of sets are disjoint and their unions are USOs. 
First, we check disjointness. 
If some $S^{(0)}_j=V_m$ and $S^{(2)}_{j'}=V_{m'}$ were not disjoint, this would imply that there are two strings in $K$ which only differ in the last coordinate, but not by exactly $2$, as $m$ and $m'$ cannot differ by exactly $2$. This would be a contradiction to $K$ being an USO.

Let us now check for the USO property. As $\big(S^{(0)}_1\cup S^{(2)}_1\big)$ forms the upper $n$-facet of $K$, and $\big(S^{(0)}_2\cup S^{(2)}_2\big)$ forms the lower $n$-facet of $K$, it is clear that these are USOs. 
For the remaining pairs, consider the orientation obtained by a partial swap as defined in \Cref{sec:nonCubicSubgraphs} to $K$ on coordinate~$n$. 
By \Cref{thm:swappingNonCubicSubgraphs}, this results in an USO. This USO has  
\(\big(S^{(0)}_1 \cup S^{(2)}_2\big)\) and \(\big(S^{(0)}_2 \cup S^{(2)}_1\big)\) as its $n$-facets, and we conclude that these are also USOs.

We have thus proven that all  combinations of $S^{(0)}_j$ and $S^{(2)}_{j'}$ are disjoint and their union is an USO, showing that our given sets $S^{(0)}_{1,2},S^{(1)}_{1,2},S^{(2)}_{1,2},S^{(3)}_{1,2}$ form a valid generalized rewriting rule.
\end{proof}

These generalized rewriting rules are strictly stronger than product constructions. 
Not only can they produce every USO as the above proof shows, we explicitly illustrate this strength by the following example, which cannot be obtained by a product construction.

\begin{example}

The following $3$-dimensional USO cannot be constructed by a non-trivial product construction ($d\not=0$ and $k\not=0$).

\begin{center}
\renewcommand{\vColorA}{red}
\renewcommand{\vColorB}{blue}
    \begin{tikzpicture}[scale=1.5]

\node(00) at (0, 0) {\textcolor{\vColorA}{0}1};
\node(01) at (0, 1) {\textcolor{\vColorB}{0}3};
\node(10) at (1, 0) {\textcolor{\vColorA}{2}0};
\node(11) at (1, 1) {\textcolor{\vColorB}{2}2};
\draw[very thick,postaction=-<-] (00) -- (10);
\draw[very thick,postaction=->-] (00) --  (01);
\draw[very thick,postaction=-<-] (10) --  (11);
\draw[very thick,postaction=-<-] (01) -- (11);
\path[thick,<->] (0,1.4) edge node[above] {$h=1$} (1,1.4);
\node at (0.5, 0.5) {bow};
\node at (2.8, 0.5) {$\Rightarrow$};
\node at (2.8, 1) {Apply rewriting rule};
\end{tikzpicture}
\begin{tikzpicture}[scale=2]

\node (101) at (0, 0, 0) {\textcolor{\vColorA}{10}1};
\node (103) at (0, 0, -1) {\textcolor{\vColorB}{10}3};
\node (020) at (0, 1, 0) {\textcolor{\vColorA}{02}0};
\node (122) at (0, 1, -1) {\textcolor{\vColorB}{12}2};

\node (220) at (1, 1, 0) {\textcolor{\vColorA}{22}0};
\node (300) at (1, 0, 0) {\textcolor{\vColorA}{30}0};
\node (312) at (1, 0, -1) {\textcolor{\vColorB}{31}2};
\node (332) at (1, 1, -1) {\textcolor{\vColorB}{33}2};

\draw[very thick,postaction=->-] (101) --  (300);
\draw[very thick,postaction=->-] (103) --  (312);
\draw[very thick,postaction=->-] (122) --  (332);
\draw[very thick,postaction=-<-] (020) --  (220);

\draw[very thick,postaction=-<-] (101) --  (020);
\draw[very thick,postaction=-<-] (103) --  (122);
\draw[very thick,postaction=-<-] (300) --  (220);
\draw[very thick,postaction=->-] (312) --  (332);

\draw[very thick,postaction=-<-] (020) --  (122);
\draw[very thick,postaction=->-] (101) --  (103);
\draw[very thick,postaction=-<-] (300) --  (312);
\draw[very thick,postaction=-<-] (220) --  (332);

\end{tikzpicture}

\end{center}
It can however be constructed by a generalized rewriting rule applied to the bow, which we will now reverse engineer: We sort the prefixes of the vertices in the target USO according to their last digit.
The four blue vertices in target USO are sorted into $\big(S^{(0)}_1 \cup S^{(2)}_1\big)$ and the four orange vertices in the front are sorted into $\big(S^{(0)}_2 \cup S^{(2)}_2\big)$. 
The vertices incident to an upwards edge are sorted into $S^{(0)}_j$, the vertices incident to a downwards edge are put into $S^{(2)}_j$.

We thus get the following generalized rewriting rule with $d=2$. When applied to the bow with $h=1$ and the labelling function $L$ as described in \cref{thm:universality}, the result is the wanted target USO.

\begin{minipage}{0.45\textwidth}
\begin{center}
\begin{tabular}{c|c|c|c}
$S^{(0)}_1$ & $S^{(2)}_1$ & $S^{(0)}_2$ & $S^{(2)}_2$\\ \hline
10 & 12 & 10 & 02 \\
   & 33 &    & 22 \\
   & 31 &    & 30 \\
\end{tabular}
\end{center}
\end{minipage}\hfill
\begin{minipage}{0.45\textwidth}
\renewcommand{\vColorA}{black}
\renewcommand{\vColorB}{black}
\begin{tikzpicture}[scale=1]
\drawUSO{1}{0}{1}{2}{3}{1}{3}{3}
\node at (1, 1) {\textcolor{blue}{$S^{(0)}_1 \cup S^{(2)}_1$}};
\end{tikzpicture}
\begin{tikzpicture}[scale=1]
\drawUSO{1}{0}{0}{2}{3}{0}{2}{2}
\node at (1, 1) {\textcolor{red}{$S^{(0)}_2 \cup S^{(2)}_2$}};
\end{tikzpicture}
\end{minipage}

\end{example}

\section{Relationship with Known Constructions}\label{sec:emulating}
In this section, we revisit the known construction methods and analyze their relationship with (generalized) rewriting rules.

\subsection{Product Construction}
We can emulate the product construction using generalized rewriting rules. The $k$-dimensional frame of the product construction is used as the input USO. Let the labelling function $L$ be a bijection from $V(Q_k)$ to $[2^k]$. The rewriting rule is applied to the last dimension, i.e., $h=k$. The set $S^{(m)}_j$ is given by the hypervertex USO $O_{L^{-1}(j)}$, where each string is additionally prefixed with the digit $m$. It is easy to see that this generalized rewriting rule maps each vertex to the corresponding hypervertex of the product construction.

\subsection{Taking Facets}
The upper $h$-facet of an input USO can be obtained by applying the rewriting rule $S^{(0)}=S^{(1)}=\emptyset$ and $S^{(2)}=S^{(3)}=\{\epsilon\}$ to dimension $h$, where $\epsilon$ is the empty string. Symmetrically, the lower $h$-facet can be obtained. Note that these rules fit the conditions of \Cref{def:rewritingRule}, as $\{\epsilon\}$ is a (in fact, the only) valid $0$-dimensional USO.

\subsection{Inherited Orientations}
Similarly to taking facets, the inherited orientation collapsing each $1$-dimensional face spanned by dimension $h$ can be achieved by applying the rewriting rule $S^{(1)}=S^{(2)}=\emptyset$ and $S^{(0)}=S^{(3)}=\{\epsilon\}$ to dimension $h$. To collapse multiple dimensions at once, this rule can be applied consecutively to these dimensions.

\subsection{Flipping and Mirroring}
We have seen in \Cref{sec:flipping,sec:swapping} that we can flip all edges in a single dimension, or mirror the USO along a single dimension using simple rewriting rules.

\subsection{Hypervertex Replacements}
A hypervertex replacement is an operation that is specific to certain USOs.
The same face cannot be replaced by a different orientation on every USO, only on those that fulfill the condition of this face being a hypervertex.
Since rewriting rules apply to every USO, without condition, something about the rewriting rule would have to be able to distinguish details about the structure of the USO.
We think that such precision is not possible.

We can, however, reprove the validity of hypervertex replacement.
Given an $n$-dimensional USO~\(K\) and a hypervertex \(H\) of \(K\), we ask if there is a dimension which is not included in the hypervertex.
If there is no such dimension, the hypervertex is actually the whole USO, and replacing it by a different USO certainly yields an USO.
If there is such a dimension, without loss of generality this dimension is \(n\).
We can view \(K\) as the result of our universality construction.
Because \(H\) is a hypervertex, all edges in direction \(n\) incident to \(H\) are oriented consistently, guaranteeing that all vertices of \(H\) belong to exactly one of \(S^{(0)}_1,S^{(0)}_2,S^{(2)}_1,S^{(2)}_2\).
Without loss of generality, assume~\(H \subseteq S^{(0)}_1\).
If we replace \(H\) with another USO, we obtain a new set, \({S'}^{(0)}_1\).
To ensure that replacing \(S^{(0)}_1\) with \({S'}^{(0)}_1\) results in an USO, we only need to check that two orientations of one dimension lower, \(\big({S'}^{(0)}_1 \cup S^{(2)}_1\big)\) and \(\big({S'}^{(0)}_1 \cup S^{(2)}_2\big)\), are USOs.
In each of these, the replacement for~\(H\) is again a hypervertex.
Applying induction on these smaller dimensional USOs, we see that \(H\) can be replaced by any other USO of the same size.

\subsection{Phase Flipping}
While we cannot emulate phase flips using (generalized) rewriting rules since the framework cannot efficiently determine phases of a given USO, phase flips have a deep connection with the splittings we used in our universality construction and with the partial swap operation from \Cref{sec:nonCubicSubgraphs}.

In the universality construction, the directions of the $n$-edges decide which vertices of the upper facet are placed into $S_1^{(0)}$, and which into $S_1^{(2)}$ (and similarly, how the vertices of the lower facet are split into $S_2^{(0)}$ and $S_2^{(2)}$). The $n$-phases, being the sets of $n$-edges which can be flipped together, thus also describe how these facets can be split, while still fulfilling that the disjoint unions $\big(S_1^{(0)}\cup S_2^{(2)}\big)$ and $\big(S_2^{(0)}\cup S_1^{(2)}\big)$ are USOs. This view allows us to prove the following lemma on partial swaps and phase flips:
\begin{lemma}\label{lem:partialSwapPreservesPhases}
Let $K$ be some $k$-dimensional USO. Performing a partial swap on $K$ in dimension $h\in [k]$ does not change the set of $h$-phases.
\end{lemma}
\begin{proof}
Let $K$ be a $k$-dimensional USO and $P$ some union of $h$-phases of $K$. Without loss of generality, assume $h=k$. We will prove that after applying a partial swap in dimension $h$ to $K$, $P$ can still be flipped. This shows that unions of $h$-phases before the partial swap remain unions of $h$-phases after the swap. As the partial swap is an involution, this proves the lemma.

We introduce the following four orientations: $K$ (our input USO), $K_f$ ($K$ after flipping $P$), $K_s$~($K$ after performing the partial swap in dimension $h$), and finally $K_{f\circ s}$, which is the orientation obtained by flipping $P$ in $K_s$. By the definition of phases, and by \Cref{thm:swappingNonCubicSubgraphs}, $K$, $K_s$, and $K_f$ are USOs. We need to prove that $K_{f\circ s}$ is also an USO.

We split each of these orientations into the four lists $V_3,V_2,V_1,V_0$ as in our universality construction (cf. \Cref{thm:universality}). If these four lists yield a valid generalized rewriting rule when used as the sets $S_1^{(0)}, S_1^{(2)}, S_2^{(0)}, S_2^{(2)}$, then the orientation must be USO, as applying the rule to the bow yields that orientation. We describe the contents of the lists $V_3,V_2,V_1,V_0$ of $K$ by the four lists $A,B,C,D$. Each of these lists is further split into two sublists, those vertices incident to an edge of~$P$, denoted by the superscript $P$, and those which are not, denoted by the superscript $\overline{P}$. We then use these $8$ sublists to characterize the other three orientations, $K_f$, $K_s$, and $K_{f\circ s}$. See \Cref{fig:phaseswap} for all of these splittings.

\begin{figure}[h!]
    \centering
    \includegraphics[scale=0.7]{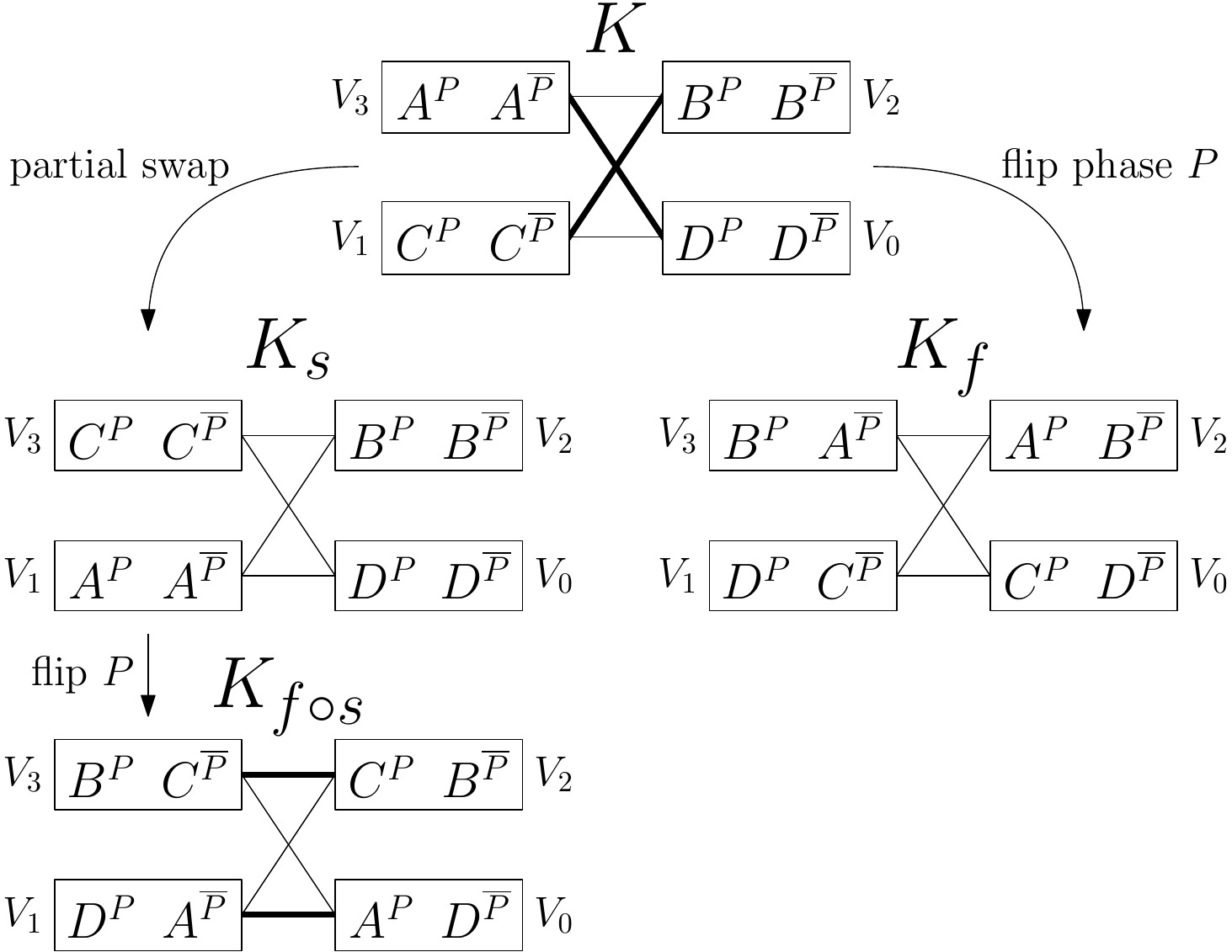}
    \caption{The splittings of the orientations $K$, $K_s$, $K_f$, and $K_{f\circ s}$ into their lists $V_3,V_2,V_1,V_0$. The connections between blocks indicate the requirement of those blocks being disjoint and forming an USO as their union. For $K$, $K_s$, and $K_f$, this is guaranteed, while we aim to prove it for $K_{f\circ s}$.}
    \label{fig:phaseswap}
\end{figure}

We now only need to prove for all four combinations of lists of $K_{f\circ s}$ (connections between blocks in \Cref{fig:phaseswap}), the combinations form an USO as their disjoint union. For the horizontal combinations (bold), this simply follows from the diagonal combinations in $K$ (bold). For the diagonal combinations, we have to do more work. We only prove it here for $D^P\cup A^{\overline{P}}\cup C^P\cup B^{\overline{P}}$, but the proof for the other combination works symmetrically. Every combination of two of these four sets occurs in some other orientation $K$, $K_f$, or $K_s$. This shows that any two of these four sets are disjoint, and they form part of an USO together, i.e., any two strings contained in them must have a coordinate in which they differ by exactly $2$. Therefore, in the union of all four sets, any two strings must differ by exactly $2$ in some coordinate. Thus, all four sets must be disjoint and form an USO together. Together with the symmetric argument for the other combination, this proves that $K_{f\circ s}$ is an USO, and that $P$ is a union of phases in $K_s$, proving the lemma.
\end{proof}

This lemma allows us to generalize the partial swaps to \emph{phase swaps}:
\begin{corollary}
Let $K$ be a $k$-dimensional USO and $h\in[k]$ some dimension. Let $P$ be some union of $h$-phases of $K$. Let $U\subset V(K)$ be the set of vertices in the upper $h$-facet of $K$ that are incident to an edge of $P$. Similarly, let $L\subset V(K)$ be the neighbors of $U$ in the lower $h$-facet. Now let $K_U$ be the subgraph induced by $U$ and $K_L$ be the subgraph induced by $L$. Swapping these two subgraphs yields another orientation $K'$, which is an USO.
\end{corollary}
\begin{proof}
This operation is equivalent to flipping $h$-phases such that $P$ is exactly the set of upwards $h$-edges of $K$, then performing a partial swap, and finally flipping the same $h$-phases again. By \Cref{lem:partialSwapPreservesPhases}, flipping the same $h$-phases after the partial swap is possible, as they are still $h$-phases.
\end{proof}

Note that while we used rewriting rules in the proof of \Cref{lem:partialSwapPreservesPhases}, the operation of a phase swap does not seem to be realizable by a (generalized) rewriting rule.

\section{Open Questions}\label{sec:conclusion}
\paragraph{Multi-Dimensional Rewriting.} The 8-dimensional counterexample to Keller's conjecture of Mackey~\cite{mackey2002eightdimensional} is built using a further generalization of rewriting rules. Instead of replacing each coordinate of a string individually, Mackey replaces pairs of coordinates together, using the information of both coordinates to pick the correct list of replacement strings. This requires $4^2=16$ lists $S^{(00)},S^{(01)},\ldots,S^{(32)},S^{(33)}$. In this setting it becomes more complex to phrase the requirements on these lists to ensure that the rewriting rule preserves the USO property. 

One possibility is to require all $4$ lists that correspond to a $2$-dimensional USO (e.g., $S^{(01)}$, $S^{(20)}$, $S^{(03)}$ and $S^{(22)}$) to form a partition of the strings describing a $d$-dimensional USO. 
This generalizes badly to higher dimensions $\ell$, where $\ell$ is the number of entries which are replaced, as one would have to know the set of all USOs of that dimension to even check that the given lists are valid.

Alternatively, one could be more strict with the requirements, using the two following conditions.
Firstly, if \(u,v\in \{0,1,2,3\}^\ell\) differ by \(2\) in some coordinate, then $S^{(u)},S^{(v)}$ are disjoint and the union $\big(S^{(u)} \cup S^{(v)}\big)$ is part of a $d$-dimensional USO.
Secondly, $\abs{S^{(m)}} = 2^d/2^\ell$ for all $m \in \{0,1,2,3\}^\ell$.

These two conditions are what was used to argue for correctness of the concrete counterexample of Mackey~\cite{mackey2002eightdimensional}. 
One can prove easily that any rewriting rule fulfilling these conditions preserves the USO property. 
While these stronger conditions are simpler to check, they have a significant drawback: it is not clear whether for every choice of $\ell$ and $d$, there even exist interesting sets of lists beyond the trivial sets one can come up with quickly.

It would be interesting to further investigate this generalization, and to research proper requirements which are easy to check and still allow for interesting instantiations.

\paragraph{Interesting Special Cases.}
While we already pointed out partial swaps (\Cref{thm:swappingNonCubicSubgraphs}) as one interesting construction arising from the rewriting rules (for $d=1$ it is the only interesting one), we strongly believe that more interesting special cases can be found for larger values of $d$. The usefulness of such concrete construction methods should be evaluated for various applications, for example for deriving lower bounds for the \textsc{Random Facet} algorithm.

\paragraph{Enumerating All USOs.}
Even though generalized rewriting rules can generate every USO, it is unclear whether they can be used to enumerate all USOs efficiently. The sequence of rewritings used in the proof of \Cref{thm:universality} is not useful for enumeration, as testing the conditions of a generalized rewriting rule is just as computationally intense as testing the resulting orientation for the USO property. Our construction might still be able to enumerate efficiently via a different sequence of constructions, and we consider finding such a sequence an important open problem.

\paragraph{Other Construction Methods.}
As can be seen from our problems with enumeration, the fact that the rewriting rules can be used to create every USO does not satisfy all of our needs for USO constructions. It is still important to research new construction methods, which may be more useful for certain applications such as enumeration. Our rewriting framework is quite different to the previously known constructions, as it is the first one based on mappings of the information of single vertices to the information of a set of vertices. Our results may open the door for more constructions of this type.

\newpage
\bibliographystyle{plainurl}
\bibliography{USO,literature}

\end{document}

%% file: includes.tex
\usepackage[T1]{fontenc}
\usepackage[english]{babel}
\usepackage[utf8]{inputenc}
\usepackage{amsmath,amssymb,amsfonts,mathrsfs,amsthm}
\usepackage{graphicx}
\usepackage{pdfpages}
\usepackage{wrapfig}
\usepackage{tikz}
\usetikzlibrary{decorations.markings}
\usetikzlibrary{decorations.pathmorphing}

\usepackage{xspace}
\usepackage{amsfonts} 
\usepackage{multirow}
\usepackage{varioref}
\usepackage{datetime}
\usepackage{mathtools}
\usepackage{ifthen}
\usepackage{stmaryrd}
\usepackage[h]{esvect}
\usepackage{array}
\usepackage{listings}
\usepackage{booktabs}
\usepackage[linesnumbered,lined,ruled,vlined]{algorithm2e}
\usepackage{makecell}
\usepackage{pdflscape}
\usepackage{subcaption}
\usepackage{thm-restate}
\usepackage[linkcolor=black,colorlinks=true,citecolor=black,filecolor=black]{hyperref}
\usepackage[capitalise]{cleveref}
\usepackage{complexity}
\usepackage{todonotes}

\usepackage{color}

\definecolor{orange}{rgb}{0.898, 0.621, 0.0}
\definecolor{skyblue}{rgb}{0.336, 0.703, 0.910}
\definecolor{green}{rgb}{0, 0.617, 0.449}
\definecolor{yellow}{rgb}{0.937, 0.890, 0.258}
\definecolor{blue}{rgb}{0, 0.445, 0.695}
\definecolor{red}{rgb}{0.832, 0.367, 0}
\definecolor{purple}{rgb}{0.797, 0.473, 0.652}

\numberwithin{equation}{section}

\newtheorem{theorem}{Theorem}
\newtheorem{example}[theorem]{Example}

\newtheorem{corollary}[theorem]{Corollary}
\newtheorem{lemma}[theorem]{Lemma}

\newtheorem{observation}[theorem]{Observation}

\theoremstyle{definition}
\newtheorem{definition}[theorem]{Definition}

\crefname{claim}{claim}{claims}
\crefname{observation}{observation}{observations}

\DeclarePairedDelimiter\abs{\lvert}{\rvert}

\renewcommand{\epsilon}{\ensuremath\varepsilon}

\renewcommand{\phi}{\ensuremath{\varphi}}

\renewcommand{\epsilon}{\ensuremath{\varepsilon}}

\renewcommand{\theta}{\ensuremath{\vartheta}}

\newcommand{\szabo}{Szab{\'o}}
\newcommand{\hajos}{Haj{\'o}s}
\newcommand{\corradi}{Corr{\'a}di}

%% file: macros.tex
\newcommand{\orientation}{O}
\newcommand{\xor}{\oplus}
\newcommand{\Cube}{Q}

\newcommand{\Reals}{\mathbb{R}}
\newcommand{\Integer}{\mathbb{Z}}
\newcommand{\Clique}{K}
\newcommand{\dimensionK}{k}
\newcommand{\dimension}{d}

\theoremstyle{definition}

\tikzset{
    ->-/.style={draw=yellow,line width=7pt, 
      postaction={draw=black,very thick,postaction={decorate,decoration={
        markings,
        mark=at position .5 with {\arrow{>}}
      }},}},
    -<-/.style={postaction={decorate,decoration={
        markings,
        mark=at position .5 with {\arrow{<}}
      }}},
}

\newcommand{\fillSingleTile}[3]{
\ifx3#1  \def\myModX{4} \else \def\myModX{5} \fi
\ifx3#2  \def\myModY{4} \else \def\myModY{5} \fi

\fill[#3] (#1, #2) -- (#1 + 1,#2) -- (#1 + 1,#2+1) -- (#1,#2 +1) -- (#1, #2);
\fill[#3] ({Mod(#1 + 1, \myModX)}, #2) -- ({Mod(#1 + 2, \myModX)},#2) -- ({Mod(#1 + 2, \myModX)},#2+1) -- ({Mod(#1 + 1, \myModX)},#2 +1) --({Mod(#1 + 1, \myModX)}, #2);
\fill[#3] ({Mod(#1 + 1,\myModX)}, {Mod(#2 + 1, \myModY)}) -- ({Mod(#1 + 2, \myModX)},{Mod(#2 + 1, \myModY)}) -- ({Mod(#1 + 2, \myModX)},{Mod(#2 + 2, \myModY)}) -- ({Mod(#1 + 1, \myModX)},{Mod(#2 + 2, \myModY)}) -- ({Mod(#1 + 1, \myModX)}, {Mod(#2 + 2, \myModY)});
\fill[#3] (#1, {Mod(#2 + 1, \myModY)}) -- ({Mod(#1 + 1, 5)},{Mod(#2 + 1, \myModY)}) -- ({Mod(#1 + 1, 5)},{Mod(#2 + 2, \myModY)}) -- (#1,{Mod(#2 + 2, \myModY)}) -- (#1, {Mod(#2 + 1, \myModY)});
}

\newcommand{\drawTiling}[8]{
\fillSingleTile{#1}{#2}{red}
\fillSingleTile{#3}{#4}{blue}
\fillSingleTile{#5}{#6}{green}
\fillSingleTile{#7}{#8}{yellow}
\draw (0,0) -- (0, 4) -- (4, 4) -- (4, 0) -- (0,0);
\fill (#1, #2)  circle[radius=2pt] node[anchor = south west]{#1#2};
\fill (#3, #4)  circle[radius=2pt] node[anchor = south west]{#3#4};
\fill (#5, #6)  circle[radius=2pt] node[anchor = south west]{#5#6};
\fill (#7, #8)  circle[radius=2pt] node[anchor = south west]{#7#8};

    \phantom{
        \fill (0, 0) circle[radius=2pt];
        \fill (4, 4) circle[radius=2pt];
    }
}

\newcommand{\vColorA}{black}
\newcommand{\vColorB}{black}
\newcommand{\vColorC}{black}
\newcommand{\vColorD}{black}
\newcommand{\drawUSO}[8]{
\node (#1#2) at (0, 0) {\textcolor{\vColorA}{#1}#2};
\node (#3#4) at (0, 2) {\textcolor{\vColorB}{#3}#4};
\node (#5#6) at (2, 0) {\textcolor{\vColorC}{#5}#6};
\node (#7#8) at (2, 2) {\textcolor{\vColorD}{#7}#8};
\ifx1#1
\draw[very thick,postaction=->-] (#1#2) -- (#5#6);
\else 
\draw[very thick,postaction=-<-] (#1#2) -- (#5#6);
\fi
\ifx1#2
\draw[very thick,postaction=->-] (#1#2) -- (#3#4);
\else 
\draw[very thick,postaction=-<-] (#1#2) -- (#3#4);
\fi
\ifx3#7
\draw[very thick,postaction=->-] (#5#6) -- (#7#8);
\else 
\draw[very thick,postaction=-<-] (#5#6) -- (#7#8);
\fi
\ifx3#8
\draw[very thick,postaction=->-] (#3#4) -- (#7#8);
\else 
\draw[very thick,postaction=-<-] (#3#4) -- (#7#8);
\fi
}

\makeatletter
\newcommand{\xRrightarrow}[2][]{\ext@arrow 0359\Rrightarrowfill@{#1}{#2}}
\newcommand{\Rrightarrowfill@}{\arrowfill@\equiv\equiv\Rrightarrow}
\newcommand{\xLleftarrow}[2][]{\ext@arrow 3095\Lleftarrowfill@{#1}{#2}}
\newcommand{\Lleftarrowfill@}{\arrowfill@\Lleftarrow\equiv\equiv}
\newcommand{\xLleftRrightarrow}[2][]{\ext@arrow 3399\LleftRrightarrowfill@{#1}{#2}}
\newcommand{\LleftRrightarrowfill@}{\arrowfill@\Lleftarrow\equiv\Rrightarrow}
\makeatother

%% file: main.bbl
\begin{thebibliography}{10}

\bibitem{bosshard2017pseudo}
Vitor Bosshard and Bernd Gärtner.
\newblock Pseudo unique sink orientations, 2017.
\newblock \href {http://arxiv.org/abs/1704.08481} {\path{arXiv:1704.08481}}.

\bibitem{brakensiek2022kellerresolved}
Joshua Brakensiek, Marijn Heule, John Mackey, and David Narv{\'a}ez.
\newblock The resolution of keller's conjecture.
\newblock {\em Journal of Automated Reasoning}, 66(3):277--300, Aug 2022.
\newblock \href {https://doi.org/10.1007/s10817-022-09623-5}
  {\path{doi:10.1007/s10817-022-09623-5}}.

\bibitem{corradi1990kellerconjecturestrings}
Keresztély Corr{\'a}di and S{\'a}ndor Szab{\'o}.
\newblock A combinatorial approach for keller's conjecture.
\newblock {\em Periodica Mathematica Hungarica}, 21(2):95--100, Jun 1990.
\newblock \href {https://doi.org/10.1007/BF01946848}
  {\path{doi:10.1007/BF01946848}}.

\bibitem{fearnley2020ueopl}
John Fearnley, Spencer Gordon, Ruta Mehta, and Rahul Savani.
\newblock Unique end of potential line.
\newblock {\em Journal of Computer and System Sciences}, 114:1--35, 2020.
\newblock \href {https://doi.org/10.1016/j.jcss.2020.05.007}
  {\path{doi:10.1016/j.jcss.2020.05.007}}.

\bibitem{gao2020dcubes}
Yuan Gao, Bernd G{\"a}rtner, and Jourdain Lamperski.
\newblock A new combinatorial property of geometric unique sink orientations,
  2020.
\newblock \href {http://arxiv.org/abs/2008.08992} {\path{arXiv:2008.08992}}.

\bibitem{gaertner2002simplex}
Bernd G{\"a}rtner.
\newblock The random-facet simplex algorithm on combinatorial cubes.
\newblock {\em Random Structures \& Algorithms}, 20(3):353--381, 2002.
\newblock \href {https://doi.org/10.1002/rsa.10034}
  {\path{doi:10.1002/rsa.10034}}.

\bibitem{gaertner2008grids}
Bernd G{\"a}rtner, Walter~D. Morris~jr., and Leo R{\"u}st.
\newblock Unique sink orientations of grids.
\newblock {\em Algorithmica}, 51(2):200--235, 2008.
\newblock \href {https://doi.org/10.1007/s00453-007-9090-x}
  {\path{doi:10.1007/s00453-007-9090-x}}.

\bibitem{gaertner2006lpuso}
Bernd G{\"a}rtner and Ingo Schurr.
\newblock Linear programming and unique sink orientations.
\newblock In {\em Proc. 17th Annual ACM-SIAM Symposium on Discrete Algorithms
  (SODA)}, pages 749--757, 2006.
\newblock \href {https://doi.org/10.5555/1109557.1109639}
  {\path{doi:10.5555/1109557.1109639}}.

\bibitem{gaertner2015recognizing}
Bernd G{\"a}rtner and Antonis Thomas.
\newblock The complexity of recognizing unique sink orientations.
\newblock In Ernst~W. Mayr and Nicolas Ollinger, editors, {\em 32nd
  International Symposium on Theoretical Aspects of Computer Science (STACS
  2015)}, volume~30 of {\em Leibniz International Proceedings in Informatics
  (LIPIcs)}, pages 341--353, Dagstuhl, Germany, 2015. Schloss
  Dagstuhl--Leibniz-Zentrum für Informatik.
\newblock \href {https://doi.org/10.4230/LIPIcs.STACS.2015.341}
  {\path{doi:10.4230/LIPIcs.STACS.2015.341}}.

\bibitem{gaertner2016niceusos}
Bernd G{\"a}rtner and Antonis Thomas.
\newblock The niceness of unique sink orientations.
\newblock In Klaus Jansen, Claire Mathieu, Jos{\'e} D.~P. Rolim, and Chris
  Umans, editors, {\em Approximation, Randomization, and Combinatorial
  Optimization. Algorithms and Techniques (APPROX/RANDOM 2016)}, volume~60 of
  {\em Leibniz International Proceedings in Informatics (LIPIcs)}, pages
  30:1--30:14, Dagstuhl, Germany, 2016. Schloss Dagstuhl--Leibniz-Zentrum für
  Informatik.
\newblock \href {https://doi.org/10.4230/LIPIcs.APPROX-RANDOM.2016.30}
  {\path{doi:10.4230/LIPIcs.APPROX-RANDOM.2016.30}}.

\bibitem{gaertner2001enforcing}
Bernd G{\"a}rtner and Emo Welzl.
\newblock Explicit and implicit enforcing - randomized optimization.
\newblock In {\em Computational Discrete Mathematics: Advanced Lectures}, pages
  25--46. Springer Berlin Heidelberg, 2001.
\newblock \href {https://doi.org/10.1007/3-540-45506-X_3}
  {\path{doi:10.1007/3-540-45506-X_3}}.

\bibitem{hajos1950factorisation}
Gy{\"o}rgy Haj{\'o}s.
\newblock Sur la factorisation des groupes ab{\'e}liens.
\newblock {\em Casopis Pest. Mat. Fys}, 74:157--162, 1950.

\bibitem{keller1930conjecture}
Ott-Heinrich Keller.
\newblock Über die lückenlose erfüllung des raumes mit würfeln.
\newblock {\em Journal für die reine und angewandte Mathematik},
  1930(163):231--248, 1930.
\newblock \href {https://doi.org/10.1515/crll.1930.163.231}
  {\path{doi:10.1515/crll.1930.163.231}}.

\bibitem{klaus2012phd}
Lorenz Klaus.
\newblock {\em A fresh look at the complexity of pivoting in linear
  complementarity}.
\newblock PhD thesis, ETH Z{\"u}rich, 2012.
\newblock \href {https://doi.org/10.3929/ethz-a-007604201}
  {\path{doi:10.3929/ethz-a-007604201}}.

\bibitem{lagarias1992keller}
Jeffrey~C. Lagarias and Peter~W. Shor.
\newblock Keller’s cube-tiling conjecture is false in high dimensions.
\newblock {\em Bulletin of the American Mathematical Society}, 27(2):279--283,
  1992.
\newblock \href {https://doi.org/10.1090/S0273-0979-1992-00318-X}
  {\path{doi:10.1090/S0273-0979-1992-00318-X}}.

\bibitem{mackey2002eightdimensional}
John Mackey.
\newblock A cube tiling of dimension eight with no facesharing.
\newblock {\em Discrete {\&} Computational Geometry}, 28(2):275--279, Aug 2002.
\newblock \href {https://doi.org/10.1007/s00454-002-2801-9}
  {\path{doi:10.1007/s00454-002-2801-9}}.

\bibitem{mathew2013enumeratingtilings}
K.~Ashik Mathew, Patric R.~J. {\"O}sterg{\aa}rd, and Alexandru Popa.
\newblock Enumerating cube tilings.
\newblock {\em Discrete {\&} Computational Geometry}, 50(4):1112--1122, Dec
  2013.
\newblock \href {https://doi.org/10.1007/s00454-013-9547-4}
  {\path{doi:10.1007/s00454-013-9547-4}}.

\bibitem{matousek2006numberusos}
Ji{\v r}{\'\i} Matou{\v s}ek.
\newblock The number of unique-sink orientations of the hypercube.
\newblock {\em Combinatorica}, 26(1):91--99, 2006.
\newblock \href {https://doi.org/10.1007/s00493-006-0007-0}
  {\path{doi:10.1007/s00493-006-0007-0}}.

\bibitem{perron1940kellersix}
Oskar Perron.
\newblock {\"U}ber l{\"u}ckenlose ausf{\"u}llung des n-dimensionalen raumes
  durch kongruente w{\"u}rfel.~{I~\&~II}.
\newblock {\em Mathematische Zeitschrift}, 46(1):1--26,161--180, Dec 1940.
\newblock \href {https://doi.org/10.1007/BF01181436}
  {\path{doi:10.1007/BF01181436}}.

\bibitem{schurr2004phd}
Ingo Schurr.
\newblock {\em Unique Sink Orientations of Cubes}.
\newblock PhD thesis, ETH Zürich, 2004.
\newblock \href {https://doi.org/10.3929/ethz-a-004844278}
  {\path{doi:10.3929/ethz-a-004844278}}.

\bibitem{schurr2004quadraticbound}
Ingo Schurr and Tibor Szab{\'o}.
\newblock Finding the sink takes some time: An almost quadratic lower bound for
  finding the sink of unique sink oriented cubes.
\newblock {\em Discrete Comput. Geom.}, 31(4):627--642, 2004.
\newblock \href {https://doi.org/10.1007/s00454-003-0813-8}
  {\path{doi:10.1007/s00454-003-0813-8}}.

\bibitem{stickney1978digraph}
Alan Stickney and Layne Watson.
\newblock Digraph models of {B}ard-type algorithms for the linear
  complementarity problem.
\newblock {\em Mathematics of Operations Research}, 3(4):322--333, 1978.
\newblock URL: \url{https://www.jstor.org/stable/3689630}.

\bibitem{szabo1993cubetilings}
S{\'a}ndor Szab{\'o}.
\newblock Cube tilings as contributions of algebra to geometry.
\newblock {\em Contributions to Algebra and Geometry}, 34(1):63--75, 1993.

\bibitem{szabo2001usos}
Tibor Szab{\'o} and Emo Welzl.
\newblock Unique sink orientations of cubes.
\newblock In {\em Foundations of Computer Science, 2001. Proceedings. 42nd IEEE
  Symposium on}, pages 547--555. IEEE, 2001.
\newblock \href {https://doi.org/10.1109/SFCS.2001.959931}
  {\path{doi:10.1109/SFCS.2001.959931}}.

\bibitem{weber2020randomfacet}
Simon Weber.
\newblock Unique sink orientations: Constructions and random facet.
\newblock Master's thesis, ETH Z{\"u}rich, November 2020.

\end{thebibliography}
